%% file: main.tex
\documentclass[10pt,oneside,reqno]{amsart}

\pdfoutput=1

\usepackage{afterpage}
\usepackage{amscd}
\usepackage{amsfonts}
\usepackage{amsmath}
\usepackage{amssymb}
\usepackage{amsthm}
\usepackage{appendix}
\usepackage{stmaryrd}

\usepackage{bbm}
\usepackage{bm}

\usepackage{calc}
\usepackage{caption}
\usepackage{changepage}
\usepackage{color}
\usepackage{comment}

\usepackage{dsfont}

\usepackage{enumitem}
\usepackage{etoolbox}
\usepackage{graphicx}

\usepackage{import}

\usepackage{mathbbol}
\usepackage{mathrsfs}
\usepackage{mathtools}

\usepackage{pifont}

\usepackage{relsize}

\usepackage{scalerel}
\usepackage{stackengine}
\usepackage{stackrel}
\usepackage{subcaption}

\usepackage{tikz} 
\usetikzlibrary{arrows,calc} 
\usepackage{transparent}

\usepackage{wasysym}
\usepackage{wrapfig}

\usepackage{xcolor}
\usepackage{xfrac}
\usepackage{xparse}

\DeclareSymbolFontAlphabet{\amsmathbb}{AMSb}% to have different mathbb -- we need them for greek letters if I'm not mistaken

\usepackage[linktocpage=true]{hyperref}

\mathtoolsset{showonlyrefs,showmanualtags}

\allowdisplaybreaks

\newcommand{\arXiv}[1]{\texttt{arXiv:\href{http://arxiv.org/abs/#1}{#1}}}

\newtheorem{theorem}{Theorem}
\newtheorem{proposition}[theorem]{Proposition}
\newtheorem{corollary}[theorem]{Corollary}
\newtheorem{lemma}[theorem]{Lemma}

\newtheorem{thmA}{Theorem}

\theoremstyle{definition}
\newtheorem{definition}[theorem]{Definition}
\newtheorem{remark}[theorem]{Remark}
\newtheorem{example}[theorem]{Example}

\numberwithin{theorem}{section}
\numberwithin{equation}{section}
\DeclareDocumentCommand{\faktor}{s m O{0.5} m O{-0.5}}{% \newfaktor[*]{#2}[#3]{#4}[#5] -> #2/#4
  \setbox0=\hbox{\ensuremath{#2}}% Store numerator
  \setbox1=\hbox{\ensuremath{\diagup}}% Store slash /
  \setbox2=\hbox{\ensuremath{#4}}% Store denominator
  \raisebox{#3\ht1}{\usebox0}% Numerator
  \mkern-5mu% Slash /
    {\rotatebox{-44}{\rule[#5\ht2]{0.4pt}{-#5\ht2+#3\ht0+\ht0}}}% tilted rule as a slash
  \mkern-4mu%
  \raisebox{#5\ht2}{\usebox2}% Denominator
}
\DeclareDocumentCommand{\doublefaktor}{s m O{0.5} m O{-0.5}}{% \newdoublefaktor[*]{#2}[#3]{#4}[#5] -> #2//=#4
  \setbox0=\hbox{\ensuremath{#2}}% Store numerator
  \setbox1=\hbox{\ensuremath{\diagup}}% Store slash /
  \setbox2=\hbox{\ensuremath{#4}}% Store denominator
  \raisebox{#3\ht1}{\usebox0}% Numerator
  \mkern-5mu% Space between numerator and double slashes
  {\rotatebox{-44}{\rule[#5\ht2]{0.4pt}{-#5\ht2+#3\ht0+\ht0}}}% First tilted rule as part of double slash
  \mkern-14mu% Space between slashes
  {\rotatebox{-44}{\rule[#5\ht2]{0.4pt}{-#5\ht2+#3\ht0+\ht0}}}% Second tilted rule as part of double slash
  \mkern-4mu% Space between slash and denominator
  \raisebox{#5\ht2}{\usebox2}% Denominator
}
\makeatletter
\newcommand{\pushright}[1]{\ifmeasuring@#1\else\omit\hfill$\displaystyle#1$\fi\ignorespaces}
\newcommand{\pushleft}[1]{\ifmeasuring@#1\else\omit$\displaystyle#1$\hfill\fi\ignorespaces}
\makeatother
\input pdf-trans
\newbox\qbox
\def\usecolor#1{\csname\string\color@#1\endcsname\space}
\newcommand\bordercolor[1]{\colsplit{1}{#1}}
\newcommand\fillcolor[1]{\colsplit{0}{#1}}
\newcommand\colsplit[2]{\colorlet{tmpcolor}{#2}\edef\tmp{\usecolor{tmpcolor}}%
  \def\tmpB{}\expandafter\colsplithelp\tmp\relax%
  \ifnum0=#1\relax\edef\fillcol{\tmpB}\else\edef\bordercol{\tmpC}\fi}
\def\colsplithelp#1#2 #3\relax{%
  \edef\tmpB{\tmpB#1#2 }%
  \ifnum `#1>`9\relax\def\tmpC{#3}\else\colsplithelp#3\relax\fi
}
\newcommand\outline[1]{\leavevmode%
  \def\maltext{#1}%
  \setbox\qbox=\hbox{\maltext}%
  \boxgs{Q q 2 Tr \thickness\space w \fillcol\space \bordercol\space}{}%
  \copy\qbox%
}
\newcommand\mathcalbb[2][1]{%
   \stackengine{0pt}{\outline{$\mathcal{#2}$}}{\kern.3pt\outline{$\mathcal{#2}$}}{O}{l}{F}{F}{L}}
\bordercolor{black}
\fillcolor{white}
\def\thickness{0.35}% TO CHANGE THICKNESS OF SHADOW
\newcommand{\GLN}{\operatorname{GL}_N}
\newcommand{\Gl}{\operatorname{GL}}
\newcommand{\Mat}{\operatorname{Mat}_N^*(\Bbbk)}

\newcommand{\sss}{\operatorname{S}(\A_{\natural})}
\newcommand{\SSS}{\amsmathbb{S}}
\newcommand{\SSym}{\operatorname{Sym_{\SSS}}}
\renewcommand{\O}{\mathcalbb{O}}

\newcommand{\tr}{\operatorname{tr}_N}
\newcommand{\1}{\widehat{1}}

\newcommand{\stimes}{\mathbin{\otimes_{\scalebox{0.6}{$\SSS$}}}}
\newcommand{\id}{\operatorname{id}}

\newcommand{\A}{\amsmathbb{A}}

\newcommand{\Der}{\operatorname{Der}}

\newcommand{\br}{\{\!\!\{-,-\}\!\!\}}

\newcommand\dsq[1]{\llbracket#1\rrbracket}

\newcommand{\kk}{\Bbbk}
\newcommand{\ZZ}{\amsmathbb{Z}}
\newcommand{\Rep}{\operatorname{Rep}}

\newcommand{\mbf}[1]{\mathbf{#1}} 
\newcommand{\mult}{\mathrm{m}} 
\title{Double Transposed Poisson Algebras}

\author{Maxime Fairon}
\address[M.F.]{Université Bourgogne Europe, CNRS, IMB UMR 5584, F-21000 Dijon, France}
\email{maxime.fairon@u-bourgogne.fr}

\author{Nikita Safonkin}
\address[N.S.]{Institute of Mathematics, Leipzig University, Augustusplatz 10, 04109 Leipzig, Germany.}
\email{safonkin.nik@gmail.com}

\begin{document}

\begin{abstract}
    We introduce \emph{double transposed Poisson algebras}, a noncommutative analogue of the transposed Poisson algebras of Bai, Bai, Guo and Wu that is compatible with the Kontsevich--Rosenberg principle. We first consider a simplified version which we call \emph{id-adapted} double transposed Poisson algebras and then explore the general definition. We prove that every such structure on a unital associative algebra $\A$ is governed by a single derivation $\A\to\A\otimes\operatorname{S}(\A/[\A,\A])$. Furthermore, this induces a $\GLN$-equivariant transposed Poisson structure on each representation algebra $\A_N=\kk[\Rep_N(\A)]$. We also define $H_0$-transposed Poisson structures, the transposed counterpart of Crawley-Boevey's $H_0$-Poisson structures, and use the trace map to obtain a transposed Poisson structure on the ring of $\GLN$-invariants $\A_N^{\GLN}$.
\end{abstract}

\maketitle

%    \setcounter{tocdepth}{2}
    %\tableofcontents

\setcounter{tocdepth}{2} 

\tableofcontents
    
    \import{}{body}

%%%%%%%%%%%%%%%%%%%%%%%%%%%%%%%%%%%%%%%%%%%%%%%%%%%%%%%%%%%%%%%%%%%%%%%%%%%
%%%%%%%%%%%%%%%%%%%%%%%%%%%%%%%%%%%%%%%%%%%%%%%%%%%%%%%%%%%%%%%%%%%%%%%%%%%
%%%%%%%%%%%%%%%%%%% BIBLIO %%%%%%%%%%%%%%%%%%%%%%%%%%%%%%%%%%%%%%%%%%%%%%%%
%%%%%%%%%%%%%%%%%%%%%%%%%%%%%%%%%%%%%%%%%%%%%%%%%%%%%%%%%%%%%%%%%%%%%%%%%%%

\end{document}

%% file: body.tex
\section{Introduction} \label{Sec:Intro}

A transposed Poisson algebra, as introduced by Bai, Bai, Guo and Wu~\cite{bai2023transposed}, carries the same data as a Poisson algebra (i.e. a commutative associative product $\cdot$ and a Lie bracket $[-,-]$ on a common vector space $A$) but its compatibility axiom is the \emph{transpose} of the Leibniz rule, in which the roles of the two operations are exchanged:
    \begin{equation}
        2z\cdot[x,y]=[z\cdot x,y]+[x,z\cdot y],\hspace{20pt} x,y,z\in A.
    \end{equation}
This transposed Leibniz rule isolates a genuinely different class of algebras from the Poisson case, with its own self-dual operad and a rich supply of polynomial identities. Its prototypical examples arise from a single derivation $D\in\Der(A)$ through the bracket $[a,b]=a\cdot D(b)-D(a)\cdot b$. On a unital algebra every transposed Poisson bracket is of this form~\cite[Thm.~6]{damas2024transposed}.

The subject has been rapidly developing. We refer the reader to the survey~\cite{damas2024transposed}
for a comprehensive account of its growing literature, and to~\cite{ferreira2021half,KK2024}
for the study of transposed Poisson algebra structures on various Lie algebras---for instance, the
construction of nontrivial such structures on the Witt algebra, on oscillator Lie algebras, and on
several families of solvable Lie algebras, together with non-existence results for the Virasoro algebra, the Schr\"odinger
algebra, and the finite-dimensional semisimple Lie algebras.

The present paper is concerned with a noncommutative analogue of transposed Poisson algebras, in the sense of the \emph{Kontsevich--Rosenberg principle}~\cite{kontsevich2000noncommutative}. This principle of noncommutative geometry states that a structure on an associative algebra $\A$ deserves to be called noncommutative-geometric only if it induces the corresponding classical structure on each representation scheme $\Rep_N(\A)$, or equivalently on the (scheme theoretic) coordinate rings $\A_N:=\kk[\Rep_N(\A)]$ for all $N\geq1$ ($\kk$ is the base field). 
In Poisson geometry this principle was first realized by Van den Bergh's \emph{double Poisson brackets}~\cite{van2008double}: a double bracket $\br$ on $\A$ induces an honest Poisson bracket on every $\A_N$. A closely related notion is Crawley-Boevey's \emph{$H_0$-Poisson structure}~\cite{crawley-boevey2011poisson}, a certain class of Lie brackets on $\A_\natural:=\A/[\A,\A]$ that induces Poisson brackets on every invariant ring $\A_N^{\GLN}$. 

Motivated by Van den Bergh's double Poisson algebras, Question~36 of the survey~\cite{damas2024transposed} asks for double transposed Poisson algebras. This question was answered in abstract terms in \cite{safonkin2025double} through the notion of a double algebra over an arbitrary operad. The purpose of the present paper is to make this answer explicit in the case of the transposed Poisson operad. Let us emphasize that our formalism is built on a unital algebra, since the representation algebra $\A_N$ is itself unital. Therefore it is not relevant for us to investigate the noncommutative analogs of transposed $\delta$-Poisson algebras, since these last commutative structures are only nontrivial on unital algebras for $\delta=\frac12$ (the ``usual'' transposed Poisson case), cf. \cite{AKB2026-delta}.

\medskip 

To describe our main results,  fix an associative unital algebra $\A$ over an algebraically closed field $\kk$ of characteristic zero. 
We start with the \emph{$\id$-adapted double transposed Poisson algebra} (Definition~\ref{Def:id-trDPA}): in this situation, $\A$ is equipped with a family of $\kk$-linear operations
\[
  \dsq{-,-}^{(r,s)}\colon\A^{\otimes r}\otimes\A^{\otimes s}\longrightarrow\A^{\otimes(r+s)}\qquad(r,s\geq1)
\]
on the tensor powers of $\A$ (where $\otimes :=\otimes_\kk$), obeying certain graded analogues of the transposed Poisson axioms: cyclic skewsymmetry, a multi-Jacobi identity, and the graded transposed Leibniz rule
\[
  2\,\gamma\otimes\dsq{\alpha,\beta}^{(r,s)}=\dsq{\gamma\otimes\alpha,\beta}^{(t+r,s)}+\tau_{(12)}^{r,t}\,\dsq{\alpha,\gamma\otimes\beta}^{(r,t+s)}
\]
for $\alpha\in \A^{\otimes r}$, $\beta \in \A^{\otimes s}$, $\gamma\in \A^{\otimes t}$, where $\tau_{(12)}^{r,t}$ swaps the tensor blocks of sizes $r$ and $t$. 
Furthermore, additional compatibility conditions invisible at the commutative level are necessary to ensure that $\id$-adapted double transposed Poisson algebras satisfy the Kontsevich--Rosenberg principle, i.e., \(\A_N\) is naturally equipped with a transposed Poisson algebra structure.  

\begin{thmA}[{Proposition~\ref{Pr:DerIdAdapt} \& Theorem~\ref{Thm:RepIdAdapt}}]
    Every $\id$-adapted double transposed Poisson algebra structure on $\A$ has the following form 
    \[
        \dsq{\alpha,\beta}^{(r,s)}=\alpha\otimes D(\beta)-D(\alpha)\otimes\beta, \qquad 
        \alpha\in \A^{\otimes r},\,\, \beta \in \A^{\otimes s}, \,\, r,s\geq 1, 
    \]
    for a unique derivation $D\in\Der(\A)$, extended by the Leibniz rule to each tensor power $\A^{\otimes s}$, $s\geq 2$. 
    
    \noindent Moreover, any such structure canonically induces, for every $N\geq1$, a $\GLN$-equivariant transposed Poisson structure on the representation algebra $\A_N$.
\end{thmA}

The general definition is more involved and is most naturally formulated in terms of a certain symmetric algebra attached to $\A$ rather than its tensor powers. Set $\A_\natural=\A/[\A,\A]$. The part played by $\A_N$ in the commutative theory is taken over, at the noncommutative level, by the \emph{double coordinate ring}
\[
  \O(\A):=\SSym\big(\A_\natural\oplus\A[-1]\big)=\bigoplus_{n\geq0}\A^{\otimes n}\otimes\sss\otimes\kk[S_n],
\]
where $\sss$ is the ordinary symmetric algebra of the vector space $\A_\natural$, and $\SSym$ denotes the symmetric algebra in the category of diagonal $\SSS$-bimodules (cf. Subsection \ref{ss:Prelim}).
This graded associative algebra first appeared, under the name \emph{Fock algebra}, in the work of Ginzburg and Schedler \cite{ginzburg2010differentialoperatorsbvstructures}. 
Moreover, $\O(\A)$ plays a key role in the approach of the second author \cite{safonkin2025double} for inverting the Kontsevich-Rosenberg principle. 

A \emph{double transposed Poisson algebra} (Definition~\ref{Def:trDPA}) is then an associative unital algebra $\A$ together with an $\SSS$-bimodule map $\dsq{-,-}\colon\O(\A)\stimes\O(\A)\longrightarrow\O(\A)$ that is cyclically twisted-skewsymmetric, satisfies the twisted-Jacobi identity, obeys the following twisted-transposed Leibniz rule
\[
    2\gamma\,\dsq{\alpha,\beta}=\dsq{\gamma\alpha,\beta}+\operatorname{Ad}\!\big((12)^{|\alpha|,|\gamma|}\big)\,\dsq{\alpha,\gamma\beta},\qquad\alpha,\beta,\gamma\in\O(\A),
\]
and satisfies additional compatibility conditions required to induce operations on each $\A_N$.

\begin{thmA}[{Theorem~\ref{Thm:DerDtPA} \& Propositions~\ref{Pr:dtPA-induces-tPA-N}, \ref{Pr:DerOA-delta}, \ref{Pr:dtPA-id-adapted}}]
    Every double transposed Poisson bracket on $\A$ has the following form
    \[
        \dsq{\alpha,\beta}=\alpha\,D(\beta)-D(\alpha)\,\beta,\qquad\alpha,\beta\in\O(\A),
    \]
    for a unique derivation $D$ of $\O(\A)$ that is also an $\SSS$-bimodule map satisfying the compatibility conditions; these are in bijection with ordinary derivations $\delta\colon\A\to\A\otimes\sss$, and the double transposed Poisson bracket restricts to an $\id$-adapted double transposed Poisson bracket precisely when $\delta$ takes values in $\A\otimes\kk\mathds{1}$.

\noindent    Moreover, for every $N\geq1$, a double transposed Poisson algebra $\A$ induces a canonical $\GLN$-equivariant transposed Poisson structure on the representation algebra $\A_N$. 
\end{thmA}

Finally, we develop the transposed analogue of Crawley-Boevey's $H_0$-Poisson structures. An \emph{$H_0$-transposed Poisson structure} (Definition~\ref{Def:trH0P}) is a transposed Poisson algebra structure $\{-,-\}$ on $\sss$ such that, for every $a\in\A$, there exists an element $d_a\in(\Der(\A)\oplus\kk\cdot\id_\A)\otimes\sss$ with $\{\overline a,\overline b\}=\pi(d_a(b))$ for all $b\in\A$, where $\overline a\in\A_\natural$ denotes the class of $a$ and $\pi\colon\A\otimes\sss\to\sss$, $x\otimes f\mapsto\overline x\cdot f$. 

\begin{thmA}[{Theorem~\ref{Pr:H0P-trace-tPA} \& Propositions~\ref{Thm:0-deg-is-H0P}, \ref{Pr:dtPA-N-matches-trace}}]
    Any $H_0$-transposed Poisson structure $\{-,-\}$ on $\A$ induces, for every $N\geq1$, a transposed Poisson structure on the ring of invariants $\A_N^{\GLN}$. The restriction of any double transposed Poisson bracket on $\O(\A)$ to its degree-zero component $\sss=\O(\A)^{(0)}$ is an $H_0$-transposed Poisson structure, and the resulting bracket on $\A_N^{\GLN}$ coincides with the restriction to $\A_N^{\GLN}$ of the $\GLN$-equivariant transposed Poisson structure that $\A$ induces on $\A_N$.
\end{thmA}

The relations between the different structures involved in Theorem C can be summarized as follows: 
\begin{center}
      \begin{tikzpicture}
%%%% top row BACK
 \node   (T0) at (-3,1) {$(\A,\dsq{-,-})$};
 \node   (T1) at (3,1) {$(\A_N,[-,-]_N)$};
 \node  (B0) at (-3,-2.5) {$(\sss,\{-,-\})$};
 \node  (B1) at (3,-0.7) {$(\A_N^{\GLN},\{-,-\}_N^{\mathrm{ind}})$};
 \node  (B2) at (3,-2.5) {$(\A_N^{\GLN},\{-,-\}_N^{\mathrm{tr}})$};
% \node (Eq) at (3,-1.5) {$\mathrel{\rotatebox{90}{$=$}}$}; 
%\draw [double equal sign distance] (B1)  to node[right,font=\small] {Prop.~\ref{Pr:dtPA-N-matches-trace}} (B2);
\path[->,>=angle 90,font=\small]  
   (T0) edge node[above] {Prop.~\ref{Pr:dtPA-induces-tPA-N}} (T1) ;
\path[->,>=angle 90,font=\small]  
   (B0) edge node[below] {Thm.~\ref{Pr:H0P-trace-tPA}} (B2) ;
\path[->,>=angle 90,font=\small]  
   (T0) edge node[left] {Thm.~\ref{Thm:0-deg-is-H0P}} (B0) ;
\path[->,>=angle 90,font=\small]  
   (T1) edge node[right] {Prop.~\ref{Pr:tPA-Red}} (B1) ;
\path[-,thick,font=\small]  
   (B1) edge node[right] {Prop.~\ref{Pr:dtPA-N-matches-trace}} (B2) ;
\path[  -,thick,  transform canvas={xshift=-0.15cm}]
(B1) edge (B2);
   \end{tikzpicture}
\end{center} 

\medskip 

\textbf{Outline of the paper.} Section~\ref{Sec:Action} recalls the definition of transposed Poisson algebras and discusses group actions on them. 
Section~\ref{Sec:Motiv} develops the special case of $\id$-adapted double transposed Poisson algebras in a self-contained manner and shows that these induce transposed Poisson brackets on every $\A_N$. 
Section~\ref{Sec:Gen} contains the general definition of double transposed Poisson algebras: after recalling the diagonal $\SSS$-bimodule formalism needed to handle the double coordinate ring $\O(\A)$, we present the general Definition~\ref{Def:trDPA} and show subsequently how these structures induce transposed Poisson brackets on the representation algebras $\A_N$ and how they are governed by a single derivation. We also identify the $\id$-adapted structures inside the general ones. 
Finally, we introduce $H_0$-transposed Poisson structures (Subsection~\ref{Sec:0-degree}), which can be carried to the invariant rings $\A_N^{\GLN}$ by the trace map, and which can be obtained by restriction to $\sss$ of double transposed Poisson algebra structures.

\textbf{Acknowledgements.} Nikita Safonkin was partially supported by the European Research Council (ERC) under Grant Agreement No.~101041499. 
The work of Maxime Fairon was carried out at IMB which receives support from the EIPHI Graduate School (ANR-17-EURE-0002).

%%%%%%%%%% NEW SECTION %%%%%%%%%
%%%%%%%%%% NEW SECTION %%%%%%%%%
%%%%%%%%%% NEW SECTION %%%%%%%%%
%%%%%%%%%% NEW SECTION %%%%%%%%%
%%%%%%%%%% NEW SECTION %%%%%%%%%
%%%%%%%%%% NEW SECTION %%%%%%%%%
\section{Action and reduction of transposed Poisson algebras} \label{Sec:Action}

We first recall, following~\cite{bai2023transposed}, the basic notions of the commutative theory.

\begin{definition}[\cite{bai2023transposed}] \label{Def:trPA}
    A transposed Poisson algebra is a triple $(A,\cdot,[-,-])$, abbreviated $A$,  where $(A,\cdot)$  is a commutative associative algebra equipped with a Lie algebra structure $[-,-]:A\times A\longrightarrow A$ such that
    \begin{equation} \label{Eq:trPA}
        2z\cdot[x,y]=[z\cdot x,y]+[x,z\cdot y],\hspace{20pt} x,y,z\in A.
    \end{equation}
\end{definition}

\begin{example}[\cite{bai2023transposed}] \label{Ex:baiDer}
 If $A$ is a commutative associative algebra and $D\in \Der(A)$ is a derivation, then 
 \begin{equation} \label{Eq:baiDer}
    [a,b]_D= a\cdot D(b) - D(a) \cdot b , \qquad \forall a,b\in A 
 \end{equation}
defines a transposed Poisson structure on $A$. 
\end{example}

\begin{proposition}[\cite{damas2024transposed},Theorem~6] \label{Pr:tPA-unit}
    If $A$ is a unital commutative associative algebra and $[-,-]$ defines a structure of transposed Poisson algebra, then the Lie bracket is of the form \eqref{Eq:baiDer} with $D=[1,-] \in \Der(A)$.
\end{proposition}

We can consider group actions on transposed Poisson algebras. 

\begin{definition}
Let $A$ be a transposed Poisson algebra, $G$ be a group, 
and $G\times A \to A$, $(g,a)\mapsto g\ast a$ be a left action for the algebra structure $(A,\cdot)$. 
The transposed Poisson structure is \emph{$G$-equivariant} if 
\begin{equation} \label{Act-tPA}
    g\ast [a,b]=[g\ast a,g\ast b], \qquad \forall g\in G, \,\, a,b \in A.
\end{equation}
\end{definition}

\begin{proposition} \label{Pr:tPA-Red}
If $A$ admits a $G$-equivariant transposed Poisson structure, then the invariant subalgebra $A^G$ inherits a transposed Poisson structure. 
\end{proposition}
\begin{proof}
   It is clear that $A^G$ inherits the multiplication of $A$. For its Lie bracket, note that if $f_1,f_2\in A^G$ then for any $g\in G$,  $g\ast [f_1,f_2] = [f_1,f_2]$ by \eqref{Act-tPA}. Hence the Lie bracket $[-,-]$ restricts to $A^G$. The defining condition \eqref{Eq:trPA} is satisfied on $A^G$ since we have an inclusion $A^G \hookrightarrow A$. 
\end{proof}
Note that when a transposed Poisson structure is $G$-equivariant, it is automatically $H$-equivariant for any subgroup $H$ of $G$.

\begin{example}
    Let $A$ be equipped with a transposed Poisson structure $[-,-]_D$ defined from some $D\in \Der(A)$ as in Example \ref{Ex:baiDer}. If a group $G$ acts on $A$ and $D$ is $G$-equivariant 
    (i.e. $D(g\ast a)=g\ast D(a)$ for all $g\in G$ and $a\in A$), then the transposed Poisson structure given by $[-,-]_D$ is $G$-equivariant due to \eqref{Eq:baiDer}. 
    By Proposition \ref{Pr:tPA-Red}, $A^G$ inherits a transposed Poisson structure which is defined by the restriction $D\big|_{A^G}\in \Der(A^G)$ of $D$. 
\end{example}

\begin{example}
For $n\geq 1$ and $\underline{\lambda}\in (\kk^\times)^n$, the oscillator algebra 
$\mathcal{L}^{(n)}_{\underline{\lambda}}$ is the $(2n+2)$-dimensional vector space with basis 
$\{e_-,e_0,e_j,\check e_j\}_{j=1}^n$ equipped with the Lie bracket whose only nonzero relations are obtained from 
\begin{equation}
    [e_-,e_j] = \lambda_j \check e_j, \quad 
    [e_-,\check e_j] = -\lambda_j e_j, \quad 
    [e_j,\check e_j] = e_0, \quad j=1,\ldots,n.
\end{equation}
It is shown in \cite{KK2024} that $\mathcal{L}^{(n)}_{\underline{\lambda}}$ can be equipped with a transposed Poisson structure by considering the following multiplication for arbitrary 
$\gamma,\mu\in \kk$ and $\underline{\alpha},\underline{\beta}\in \kk^n$: 
\begin{align}
   e_-\cdot e_- &= \gamma e_- + \mu e_0 + 2\sum_{k=1}^n \lambda_k (\alpha_k e_k + \beta_k \check e_k), \\
   e_-\cdot e_0 &= \gamma e_0, \quad 
   e_-\cdot e_j = \alpha_j e_0 + \gamma e_j, \quad 
   e_-\cdot \check e_j = \beta_j e_0 + \gamma \check e_j, \\
   e_j\cdot e_j &= \check e_j\cdot \check e_j = -\frac{\gamma}{2\lambda_j}e_0\,.
\end{align}
The multiplication is unital if and only if $\underline{\alpha},\underline{\beta}$ are the zero vector and $\gamma\neq 0$, in which case the unit is $\frac{1}{\gamma}e_--\frac{\mu}{\gamma^2}e_0$.

There is an action of the group $\ZZ_2$ on $\mathcal{L}^{(n)}_{\underline{\lambda}}$ (seen first as a vector space) where its generator acts through 
$e_j \leftrightarrow \check e_j$, $e_-\to -e_-$, $e_0\to -e_0$.  
The above transposed Poisson structure is $\ZZ_2$-equivariant provided that 
$\gamma=\mu=0$ and $\alpha_j=\beta_j$ for all $j$. 
We get a transposed Poisson structure using Proposition \ref{Pr:tPA-Red} on 
$(\mathcal{L}^{(n)}_{\underline{\lambda}})^{\ZZ_2}=\oplus_j \kk (e_j+\check e_j)$ where both the Lie bracket and the multiplication are trivially zero. 

\noindent There is also an action of the group $\ZZ_n$, where a generator is required to act by 
$e_j \to e_{j+1}$, $\check e_j \to \check e_{j+1}$ while it leaves $e_-,e_0$ invariant. 
The above transposed Poisson structure is $\ZZ_n$-equivariant provided that 
all $\lambda_j$ are equal, and similarly for the components of $\underline{\alpha},\underline{\beta}$.
We then get a transposed Poisson structure using Proposition \ref{Pr:tPA-Red} on 
$(\mathcal{L}^{(n)}_{\underline{\lambda}})^{\ZZ_n}\simeq \mathcal{L}^{(1)}_{\lambda_1}$ for the following parameters in the multiplication: $(\gamma,\mu,\sqrt{n}\alpha_1,\sqrt{n}\beta_1)$. 
%% In this second identification, the new e_1 is \frac{1}{\sqrt{n}}\sum_j e_j
\end{example}

%%%%%%%%%% NEW SECTION %%%%%%%%%
%%%%%%%%%% NEW SECTION %%%%%%%%%
%%%%%%%%%% NEW SECTION %%%%%%%%%
%%%%%%%%%% NEW SECTION %%%%%%%%%
%%%%%%%%%% NEW SECTION %%%%%%%%%
%%%%%%%%%% NEW SECTION %%%%%%%%%
\section{Motivation and first definition} \label{Sec:Motiv}

In this section, we build a first definition of double transposed Poisson algebras based on the construction using derivations from Example \ref{Ex:baiDer}. 
It is meant to be pedagogical and inspired by \cite{van2008double}, so that the reader can have Definition \ref{Def:id-trDPA} in mind before encountering the general formalism developed in the next section. 

\subsection{Notation and basic constructions} \label{ss:MotivNot}

From now on, $\A$ is an associative unital algebra $\A$ over an algebraically closed field $\kk$ of characteristic zero. The multiplication $\mult : \A\otimes \A\to \A$ is simply denoted by concatenation. 
For any $r\geq 2$, the $r$-th tensor power (over $\kk$) $\A^{\otimes r}$ is naturally equipped with an associative multiplication given by factorwise multiplication. 
We also define for any $1\leq u<v\leq r$ the map 
$\mult_{u,v}:  \A^{\otimes r}\to \A^{\otimes (r-1)}$ by 
\begin{equation}  \label{Eq:mult-uv}
 \mult_{u,v}(a^1 \otimes \cdots \otimes a^r)
= a^1 \otimes \cdots \otimes a^{u-1}\otimes a^{u+1} \otimes \cdots \otimes a^{v-1} \otimes  a^u a^v \otimes a^{v+1}\otimes \cdots \otimes a^r\,,
\end{equation}
which reduces the number of tensor factors; 
we can increase the number of tensor factors from $\A^{\otimes r}$ to $\A^{\otimes (r+1)}$ for any $1\leq \rho \leq r$ through
\[
b\otimes_\rho (a^1 \otimes \cdots \otimes a^r)
:= a^1 \otimes \cdots \otimes a^{\rho} \otimes b \otimes a^{\rho+1} \otimes \cdots \otimes a^r\,.
\]
We also set $b\otimes_0 (a^1 \otimes \cdots \otimes a^r):=b\otimes a^1 \otimes \cdots \otimes a^r$. 

A derivation is an element of 
$\Der(\A)=\{D\in \operatorname{End}_\kk(\A) \mid D(ab)=a D(b) + D(a) b\}$. 
Given $D\in \Der(\A)$ and $r\geq 1$, we get a derivation of $\A^{\otimes r}$ which we also denote as $D$ and is given by 
\begin{equation} \label{Eq:Dext}
  D(a^1 \otimes \cdots \otimes a^r)  
=\sum_{1\leq s\leq r} a^1 \otimes \cdots \otimes   D(a^s) \otimes \cdots \otimes a^r\,.
\end{equation}
We form the graded vector space $\mbf T \A:=\bigoplus_{r\geq 0}\A^{\otimes r}$ where, for a homogeneous element $\alpha \in \A^{\otimes r}\subset \mbf T\A$, we let $|\alpha|=r$ denote its degree.  
Tensor multiplication is compatible with the grading since 
$|\alpha \otimes \beta|=|\alpha|+|\beta|$. 
We shall also consider the following actions of the symmetric groups $S_2,S_3$ for $\alpha,\beta,\gamma\in \mbf T\A$ of respective degrees $r,s,t\geq 1$: 
\begin{align*}
    &\tau_{(12)}^{r,s}(\alpha\otimes \beta)=\beta \otimes \alpha, \quad
    \tau_{(12)}^{r,s}(\alpha\otimes \beta\otimes \gamma)=\beta \otimes \alpha \otimes \gamma, \quad
    \tau_{(123)}^{r,s,t}(\alpha\otimes \beta\otimes \gamma)=\gamma \otimes \alpha \otimes \beta ,
\end{align*}
and we define similarly the other permutations. 

\medskip 

For an integer $N\geq 1$, the $N$-th representation algebra of $\A$ is the (scheme-theoretic) coordinate ring $\A_N=\kk[\Rep_N(\A)]$ of its $N$-th representation space. 
It can be described as the associative commutative unital algebra generated by elements $a_{ij}$ with $a\in \A$, $1\leq i,j\leq N$, subject to 
\begin{equation} \label{Eq:AN-def}
(\lambda a+\mu b)_{ij} = \lambda \, a_{ij} + \mu\, b_{ij}, \quad 
1_{ij}=\delta_{ij}, \quad 
(ab)_{ij}=\sum_{k=1}^N a_{ik}b_{kj}
\end{equation}
for $a,b\in \A$, $\lambda,\mu \in \kk$ and $1\leq i,j\leq N$. 
(In the middle equality, we use Kronecker's delta function.)
As a vector space, $\A_N$ is spanned by elements of the form 
\begin{equation} \label{Eq:spanAN}
    \alpha_{\mbf i \mbf j} := a_{i_1 j_1}^1 \, a_{i_2 j_2}^2 \cdots a_{i_r j_r}^r
\end{equation}
where $r\geq 1$, $\alpha=a^1 \otimes \cdots \otimes a^r \in \A^{\otimes r}$, and 
$\mbf i=(i_1,\ldots,i_r)$, $\mbf j=(j_1,\ldots,j_r)$ belong to $\{1,\ldots,N\}^{\times r}$. 
It is clear how the first two relations in \eqref{Eq:AN-def} can be written for an element of the form \eqref{Eq:spanAN}. 
To do so for the last relation in \eqref{Eq:AN-def}, consider 
multi-indices  $(\mbf i',\mbf j')$ of length $r-1$. Fix some $1\leq u  \leq r-1$, and for $o\in \{1,\ldots,N\}$ introduce the multi-indices $(\mbf i'_{o},\mbf j'_{o})$ where $o$ is added as an index to $\mbf i'$ in position $u+1$ and to $\mbf j'$ in position $u$, i.e. 
\begin{equation} \label{Eq:iojo}
    \mbf i'_{o}:=(i'_1,\ldots,i_u',o,i'_{u+1},\ldots,i'_{r-1}), \quad 
\mbf j'_{o}:=(j'_1,\ldots,j_{u-1}',o,j'_u,\ldots,i'_{r-1}).
\end{equation}
Then, one has the relation 
\begin{equation} \label{Eq:iojo2}
    (\mult_{u,u+1}\alpha)_{\mbf i',\mbf j'} 
= \sum_{1\leq o\leq N} \alpha_{\mbf i'_o \mbf j'_o} \,.
\end{equation}
Finally, note that any $D\in \Der(\A)$ induces a derivation on $\A_N$, denoted $D_N$, satisfying $D_N(a_{ij})=D(a)_{ij}$. It is given on elements of the form \eqref{Eq:spanAN} by  
\begin{equation} \label{Eq:DerInduce}
 D_N(\alpha_{\mbf i\mbf j})=D(\alpha)_{\mbf i\mbf j},   
\end{equation}
where on the right side we use the extension \eqref{Eq:Dext} of $D$ to $\A^{\otimes r}$. 
There is a natural left action of $\Gl_N(\kk)$ on $\A_N$ given as follows on generators:
\begin{equation} \label{Eq:Act}
    g\ast a_{ij} = \sum_{1\leq u,v \leq N} (g^{-1})_{iu} a_{uv} g_{vj}, \qquad 
g\in \Gl_N(\kk),\,\, a \in \A, \,\, 1\leq i,j \leq N\,.
\end{equation}
Then, the derivation $D_N\in \Der(\A_N)$ induced by $D\in \Der(\A)$ is automatically $\Gl_N(\kk)$-equivariant. 

\subsection{Discovering (\texorpdfstring{$\id$}{id}-adapted) double transposed Poisson algebras}

\begin{lemma}
   Let $D\in \Der(\A)$. Then, for any $N\geq 1$,  the operation 
   $[-,-]_N:\A_N\times \A_N\to \A_N$ satisfying 
for any $\alpha,\beta \in \mbf T\A$ homogeneous and 
$\mbf i, \mbf j\in \{1,\ldots,N\}^{\times |\alpha|}$, 
$\mbf k, \mbf l\in \{1,\ldots,N\}^{\times |\beta|}$, 
\begin{equation} \label{Eq:Lem1}
    [\alpha_{\mbf i \mbf j} , \beta_{\mbf k \mbf l}]_N
= \alpha_{\mbf i \mbf j} \, D(\beta)_{\mbf k \mbf l} 
- D(\alpha)_{\mbf i \mbf j} \, \beta_{\mbf k \mbf l}
\end{equation}
is a Lie bracket on $\A_N$. Furthermore, it 
turns $\A_N$ with its commutative associative product into a  transposed Poisson algebra.  
\end{lemma}
\begin{proof}
 This is an application of Example \ref{Ex:baiDer} with the derivation $D_N\in \Der(\A_N)$ induced by $D$.    
\end{proof}

\begin{example}
    Fix $d\geq 2$,  and for $\A=\kk\langle x_1,\ldots,x_d\rangle$, $m\in \{1,\ldots,d\}$ let $\partial_m\in \Der(\A)$ be uniquely determined by $\partial_m(x_j)=\delta_{mj}$. 
This entails 
$\partial_m(x_{j_1}\cdots x_{j_\ell})=\sum_{\lambda=1}^\ell \delta_{m j_\lambda} x_{j_1}\cdots x_{j_{\lambda-1}} x_{j_{\lambda+1}}\cdots x_{j_\ell}$ on a monomial. 
Given $\alpha_{\mbf i \mbf j}\in \A_N$ as in \eqref{Eq:spanAN}, one directly finds using \eqref{Eq:Dext}: 
\begin{align*} 
%\partial_m(\alpha)&=\sum_{1\leq \rho \leq r}  a^1 \otimes \cdots \otimes  a^{\rho-1} \otimes \partial_m (a^\rho) \otimes a^{ \rho+1} \otimes \cdots \otimes a^r,  \\
%\text{thus} \quad 
\partial_m(\alpha)_{\mbf i \mbf j} &=    \sum_{1\leq \rho \leq r} \partial_m (a^\rho)_{i_\rho j_\rho} \, \prod_{t\neq \rho} a_{i_t j_t}^t\,.
\end{align*}
Thus, the transposed Poisson structure on $\A_N$ obtained from $\partial_m\in \Der(\A)$ is determined by ($\beta=b_1\otimes \cdots \otimes b_s$)
\begin{equation} 
    [\alpha_{\mbf i \mbf j} , \beta_{\mbf k \mbf l}]_N
= \prod_{1\leq\rho\leq r} a_{i_\rho j_\rho}^\rho \, \sum_{1\leq \varsigma \leq s} \partial_m (b^\varsigma)_{k_\varsigma l_\varsigma} \, \prod_{t\neq \varsigma} b_{k_t l_t}^t 
-\prod_{1\leq \varsigma\leq s} b_{k_\varsigma l_\varsigma}^\varsigma  \,  \sum_{1\leq \rho \leq r} \partial_m (a^\rho)_{i_\rho j_\rho} \, \prod_{t\neq \rho} a_{i_t j_t}^t\,.
\end{equation}
For small degrees, this takes the following explicit form. Writing $\alpha=a$ or $\alpha=a^1\otimes a^2$ and $\beta=b$ or $\beta=b^1\otimes b^2$, the four cases with $|\alpha|,|\beta|\in \{1,2\}$ read
\begin{align*}
[a_{i_1 j_1} , b_{k_1 l_1}]_N
&= a_{i_1 j_1}\, \partial_m(b)_{k_1 l_1}
- b_{k_1 l_1}\, \partial_m(a)_{i_1 j_1}\,, \\
[a_{i_1 j_1} , b^1_{k_1 l_1} b^2_{k_2 l_2}]_N
&= a_{i_1 j_1}\big(\partial_m(b^1)_{k_1 l_1}\, b^2_{k_2 l_2} + b^1_{k_1 l_1}\, \partial_m(b^2)_{k_2 l_2}\big)
- b^1_{k_1 l_1} b^2_{k_2 l_2}\, \partial_m(a)_{i_1 j_1}\,, \\
[a^1_{i_1 j_1} a^2_{i_2 j_2} , b_{k_1 l_1}]_N
&= a^1_{i_1 j_1} a^2_{i_2 j_2}\, \partial_m(b)_{k_1 l_1}
- b_{k_1 l_1}\big(\partial_m(a^1)_{i_1 j_1}\, a^2_{i_2 j_2} + a^1_{i_1 j_1}\, \partial_m(a^2)_{i_2 j_2}\big)\,, \\
[a^1_{i_1 j_1} a^2_{i_2 j_2} , b^1_{k_1 l_1} b^2_{k_2 l_2}]_N
&= a^1_{i_1 j_1} a^2_{i_2 j_2}\big(\partial_m(b^1)_{k_1 l_1}\, b^2_{k_2 l_2} + b^1_{k_1 l_1}\, \partial_m(b^2)_{k_2 l_2}\big) \\
&\quad - b^1_{k_1 l_1} b^2_{k_2 l_2}\big(\partial_m(a^1)_{i_1 j_1}\, a^2_{i_2 j_2} + a^1_{i_1 j_1}\, \partial_m(a^2)_{i_2 j_2}\big)\,.
\end{align*}
\end{example}

Let us rewrite \eqref{Eq:Lem1}. 
Given multi-indices $\mbf i=(i_1,\ldots,i_r)$, $\mbf k=(k_1,\ldots,k_s)$, we let 
\[
\mbf i\sqcup \mbf k=(i_1,\ldots,i_r,k_1,\ldots,k_s) \in \{1,\ldots,N\}^{\times (r+s)}\,.
\]
After introducing the following family of mappings (with $r,s\geq 1$), 
\begin{align} \label{Eq:dsq-rs}
\llbracket-,-\rrbracket^{(r,s)}: \A^{\otimes r}\otimes \A^{\otimes s} \to \A^{\otimes (r+s)}, \qquad 
    \dsq{\alpha,\beta}^{(r,s)}:=\alpha \otimes D(\beta) - D(\alpha) \otimes \beta  
\end{align}
we can concisely rewrite \eqref{Eq:Lem1} as 
\begin{equation} \label{Eq:concis}
    [\alpha_{\mbf i \mbf j} , \beta_{\mbf k \mbf l}]_N = 
\llbracket \alpha , \beta\rrbracket^{(|\alpha|,|\beta|)}_{\mbf i\sqcup \mbf k,\mbf j\sqcup \mbf l} \,.
\end{equation}
In particular, we can compute $[-,-]_N$ on any two elements of the form \eqref{Eq:spanAN} by first computing the index-free operation \eqref{Eq:dsq-rs} and then concatenating the pair of indices. 
So we need to understand how to translate the axioms defining a transposed Poisson algebra in terms of the mappings \eqref{Eq:dsq-rs}. 

\begin{proposition} \label{Prop1}
  The family of operations $\llbracket-,-\rrbracket^{(r,s)}$ \eqref{Eq:dsq-rs} satisfies 
for $\alpha,\beta,\gamma\in \mbf T \A$ homogeneous of respective degrees $r,s,t$:
\begin{align}
&\dsq{\alpha,\beta}^{(r,s)} = - \tau_{(12)}^{s,r} \, \dsq{\beta,\alpha}^{(s,r)}; 
\label{Eq:Prop1a}\\
&\dsq{\alpha,\dsq{\beta,\gamma}^{(s,t)}}^{(r,s+t)}
   + \tau_{(123)}^{s,t,r} \dsq{\beta,\dsq{\gamma,\alpha}^{(t,r)}}^{(s,r+t)}
   + \tau_{(132)}^{t,r,s} \dsq{\gamma,\dsq{\alpha,\beta}^{(r,s)}}^{(t,r+s)}=0; 
   \label{Eq:Prop1b}\\
&2\, \gamma\otimes \dsq{\alpha,\beta}^{(r,s)}
   = \dsq{\gamma \otimes \alpha,\beta}^{(t+r,s)} + \tau_{(12)}^{r,t}\dsq{\alpha,\gamma \otimes\beta}^{(r,t+s)} ; 
   \label{Eq:Prop1c}\\
&2 \, \dsq{\alpha,\beta}^{(r,s)} \otimes \gamma 
   = \tau_{(132)}^{t,r,s}\dsq{\gamma \otimes \alpha,\beta}^{(t+r,s)} + \tau_{(23)}^{t,s}\dsq{\alpha,\gamma \otimes\beta}^{(r,t+s)} .
   \label{Eq:Prop1d}
\end{align}
These are respectively called the cyclic skewsymmetry, the multi-Jacobi identity and the (two) compatibility conditions. Furthermore, one has the compatibility with the unit 
\begin{equation} \label{Eq:Prop1unit} 
    \dsq{1,\beta}^{(1,s)} \in \A^{\otimes s}\cong \kk \otimes \A^{\otimes s} \subset \A^{\otimes (s+1)}\,,
\end{equation}
the tensor $\kk$-linearity rules, for any $0\leq \rho\leq r$ and $0\leq \varsigma \leq s$, 
\begin{equation} \label{Eq:Prop1L} 
    \dsq{1\otimes_\rho \alpha,\beta}^{(r+1,s)} = 1\otimes_\rho \dsq{\alpha,\beta}^{(r,s)}\, , \quad 
\dsq{\alpha,1\otimes_\varsigma \beta}^{(r,s+1)} = 1\otimes_{r+\varsigma} \dsq{\alpha,\beta}^{(r,s)}\,,
\end{equation}
the multiplication rules, for any $1\leq \rho\leq r-1$ and $1\leq \varsigma \leq s-1$,
\begin{equation} \label{Eq:Prop1M} 
    \dsq{\mult_{\rho,\rho+1} \alpha,\beta}^{(r-1,s)} = \mult_{\rho,\rho+1}\dsq{\alpha,\beta}^{(r,s)}\, , \quad 
\dsq{\alpha,\mult_{\varsigma,\varsigma+1} \beta}^{(r,s-1)} = \mult_{r+\varsigma,r+\varsigma+1}\dsq{\alpha,\beta}^{(r,s)}\,,
\end{equation}
and the following rules of permutation within an argument,
\begin{equation} \label{Eq:Prop1S} 
    \dsq{\alpha,\tau_{\sigma}^{m_1,\ldots,m_\ell}\beta}^{(r,s)}
= \tau_{\widehat{\sigma}}^{r,m_1,\ldots,m_\ell}  \dsq{\alpha,\beta}^{(r,s)}\,, \quad
    \dsq{\tau_{\sigma}^{m_1,\ldots,m_\ell}\beta,\alpha}^{(s,r)}
= \tau_{\sigma}^{m_1,\ldots,m_\ell}  \dsq{\beta,\alpha}^{(s,r)}\,,
\end{equation}
where $m_1,\ldots,m_\ell \in \ZZ_{>0}$, $m_1+\ldots+m_\ell=s$, and $\sigma\in S_\ell$ 
is lifted to $\widehat{\sigma}=\id_1\times \sigma\in S_{\ell+1}$.  
\end{proposition}

\begin{remark} \label{Rem:pr1}
Evaluating \eqref{Eq:Prop1b} and \eqref{Eq:Prop1d} at the pair of multi-indices $(\mbf i\sqcup \mbf k\sqcup \mbf u,\mbf j\sqcup \mbf l \sqcup \mbf v)$, we recover thanks to \eqref{Eq:concis} the Jacobi identity and the compatibility condition \eqref{Eq:trPA} under $[-,-]_N$ for the triple $\alpha_{\mbf i \mbf j}, \beta_{\mbf k \mbf l}, \gamma_{\mbf u \mbf v}$. Similarly, evaluating  \eqref{Eq:Prop1a} at  $(\mbf i\sqcup \mbf k,\mbf j\sqcup \mbf l)$ yields the skewsymmetry condition.
\end{remark}

\begin{proof}[Proof of Proposition~\ref{Prop1}]
The first equality \eqref{Eq:Prop1a} is an easy consequence of the definition of the operations \eqref{Eq:dsq-rs}. 
Next, recall that $D$ was extended to $\mbf T \A$ so that $D(\alpha \otimes \beta)=D(\alpha)\otimes \beta +\alpha \otimes D(\beta)$. We compute thanks to \eqref{Eq:dsq-rs}
\begin{align*}
    \dsq{\alpha,\dsq{\beta,\gamma}^{(s,t)}}^{(r,s+t)}
 =&\alpha \otimes D(\dsq{\beta,\gamma}^{(s,t)})
 - D(\alpha)\otimes \dsq{\beta,\gamma}^{(s,t)} \\
 =&\alpha \otimes \beta \otimes D^2(\gamma)-\alpha \otimes D^2(\beta)\otimes \gamma
 - D(\alpha)\otimes\beta \otimes D(\gamma)+D(\alpha)\otimes D(\beta)\otimes \gamma \,.
\end{align*}
One can easily derive \eqref{Eq:Prop1b} by permuting factors in this expression.
After this, note that \eqref{Eq:Prop1c} and \eqref{Eq:Prop1d} are equivalent since they only change by a permutation of tensor factors. To derive the former, we compute thanks to \eqref{Eq:dsq-rs}
\begin{align*}
&\dsq{\gamma \otimes \alpha,\beta}^{(t+r,s)} + \tau_{(12)}^{r,t}\dsq{\alpha,\gamma \otimes\beta}^{(r,t+s)} \\
=& \gamma \otimes \alpha \otimes D(\beta) - D(\gamma \otimes \alpha) \otimes \beta
+ \tau_{(12)}^{r,t} (\alpha \otimes D(\gamma \otimes \beta) - D(\alpha) \otimes \gamma \otimes\beta)\\
=&2 \gamma \otimes \alpha \otimes D(\beta) - 2 \gamma \otimes D(\alpha) \otimes \beta, 
\end{align*}
which is precisely $2\, \gamma\otimes \dsq{\alpha,\beta}^{(r,s)}$. 
The decomposition \eqref{Eq:Prop1unit} is clear from \eqref{Eq:dsq-rs}. 
For \eqref{Eq:Prop1L} and \eqref{Eq:Prop1M}, it suffices each time to check the first equality since the second one is equivalent due to \eqref{Eq:Prop1a}. One has by \eqref{Eq:dsq-rs} 
\begin{align}
 \dsq{1\otimes_\rho \alpha,\beta}^{(r+1,s)}
 =(1\otimes_\rho \alpha) \otimes D(\beta) - D(1\otimes_\rho \alpha) \otimes \beta 
= 1\otimes_\rho (\alpha \otimes D(\beta) - D(\alpha) \otimes \beta) = 
1\otimes_\rho \dsq{\alpha,\beta}^{(r,s)}
\end{align}
where we used that $D(1)=0$ and this yields \eqref{Eq:Prop1L}. Similarly,  
\begin{align}
 \dsq{\mult_{\rho,\rho+1}\alpha,\beta}^{(r-1,s)}
 =(\mult_{\rho,\rho+1}\alpha) \otimes D(\beta) - D(\mult_{\rho,\rho+1} \alpha) \otimes \beta 
= \mult_{\rho,\rho+1} (\alpha \otimes D(\beta) - D(\alpha) \otimes \beta) = 
\mult_{\rho,\rho+1}\dsq{\alpha,\beta}^{(r,s)}
\end{align}
since $D$ and $\mult_{\rho,\rho+1}$ commute, which gives \eqref{Eq:Prop1M}. 

Finally, we need to derive the first equality in \eqref{Eq:Prop1S} since the second one is equivalent due to \eqref{Eq:Prop1a}. 
We put $\beta=\beta^{(1)}\otimes \cdots \otimes \beta^{(\ell)}$, $\beta^{(j)}\in \A^{\otimes m_j}$, so that $\tau_{\sigma}^{m_1,\ldots,m_\ell}$ is simply the permutation $\sigma\in S_\ell$ applied to this factorization in $\ell$ tensor factors. Then, using \eqref{Eq:dsq-rs}, we directly get
\begin{align}
  \dsq{\alpha,\tau_{\sigma}^{m_1,\ldots,m_\ell}\beta}^{(r,s)}
 =& \alpha \otimes D(\tau_\sigma \beta^{(1)}\otimes \cdots \otimes \beta^{(\ell)})
 - D(\alpha) \otimes (\tau_\sigma \beta^{(1)}\otimes \cdots \otimes \beta^{(\ell)}) \\
=& \tau_{\widehat{\sigma}}^{r,m_1,\ldots,m_\ell}  \dsq{\alpha,\beta}^{(r,s)} \,,
\end{align}
which concludes the proof.
\end{proof}

Let us emphasize that, as part of this proof, we used that the 2 equalities in 
\eqref{Eq:Prop1L} (similarly for \eqref{Eq:Prop1M} and \eqref{Eq:Prop1S}) are equivalent under the cyclic skewsymmetry \eqref{Eq:Prop1a}. 
Moreover, \eqref{Eq:Prop1c} and \eqref{Eq:Prop1d} follow from one another by applying a permutation. Using \eqref{Eq:Prop1S}, one can also verify that they are equivalent to 
\begin{align} 
&2\, \gamma\otimes \dsq{\alpha,\beta}^{(r,s)}
   =\tau_{(12)}^{r,t} \dsq{\alpha\otimes \gamma,\beta}^{(r+t,s)} + \tau_{(123)}^{r,s,t}\dsq{\alpha,\beta\otimes \gamma}^{(r,s+t)} ; 
   \label{Eq:equiv1c}\\
&2 \, \dsq{\alpha,\beta}^{(r,s)} \otimes \gamma 
   = \tau_{(23)}^{t,s}\dsq{\alpha\otimes \gamma,\beta}^{(r+t,s)} 
   + \dsq{\alpha,\beta\otimes\gamma}^{(r,s+t)} .
   \label{Eq:equiv1d}
\end{align}

\begin{definition} \label{Def:id-trDPA} 
   An \emph{$\id$-adapted double transposed Poisson algebra} is an associative unital algebra $\A$ 
equipped with a family of $\kk$-linear maps ($r,s\geq 1$)
\begin{align} \label{Eq:id-trDPA}
\llbracket-,-\rrbracket^{(r,s)}: \A^{\otimes r}\otimes \A^{\otimes s} \to \A^{\otimes (r+s)}, 
\end{align}
satisfying the cyclic skewsymmetry \eqref{Eq:Prop1a}, the multi-Jacobi identity \eqref{Eq:Prop1b}, the compatibility conditions \eqref{Eq:Prop1c}-\eqref{Eq:Prop1d}, 
the compatibility with the unit \eqref{Eq:Prop1unit}, 
the tensor $\kk$-linearity rules \eqref{Eq:Prop1L}, the multiplication rules \eqref{Eq:Prop1M}, 
and the permutation rules \eqref{Eq:Prop1S}.  
\end{definition}

Note that \eqref{Eq:Prop1unit} can be reformulated as saying that \eqref{Eq:Prop1L} also holds for $r=0$ or $s=0$ after identifying $\kk 1\subset \A$ with $\kk$.

\begin{example}
  Fix a derivation $D\in \Der(\A)$. Then endowing $\A$ with the family of maps $\llbracket-,-\rrbracket^{(r,s)}$ \eqref{Eq:dsq-rs} defines an $\id$-adapted double transposed Poisson algebra by Proposition \ref{Prop1}. 
\end{example}

\begin{remark}
   The terminology  ``$\id$-adapted'' will become transparent to the reader after understanding the general Definition \ref{Def:trDPA}. We are currently considering the specific case where elements of $\sss$ do not appear in the bracket and where the factor in the group algebra is always the identity element.  
\end{remark}

\subsection{Main properties}

\begin{proposition} \label{Pr:DerIdAdapt}
An $\id$-adapted double transposed Poisson algebra is of the form  \eqref{Eq:dsq-rs} with
$D\in \Der(\A)$ given for any $\alpha\in \A$ by setting 
\begin{equation} \label{Eq:Dalpha}
     D(\alpha)=\mult\dsq{1,\alpha}^{(1,1)}\,.
\end{equation}
\end{proposition}
\begin{proof}
Consider the $\kk$-linear map $D:\A\to \A$ defined by \eqref{Eq:Dalpha}. 
Fix $\beta,\gamma\in \A$.  
Firstly, we derive by successive application of \eqref{Eq:Prop1M}, \eqref{Eq:Prop1c} and \eqref{Eq:Prop1L}, 
\begin{align*}
\dsq{1,\beta\gamma}^{(1,1)}
=\mult_{2,3}\dsq{1,\beta\otimes\gamma}^{(1,2)} 
=& \mult_{2,3}\circ \tau_{(12)}^{1,1}\left(2 \beta\otimes\dsq{1,\gamma}^{(1,1)} 
- \dsq{1\otimes_1 \beta,\gamma}^{(2,1)} \right) \\
=&2 (1\otimes\beta) \dsq{1,\gamma}^{(1,1)} - \mult_{2,3}\,( 1\otimes \dsq{\beta, \gamma}^{(1,1)}) \,.
\end{align*}
And similarly by applying \eqref{Eq:Prop1M}, \eqref{Eq:equiv1d}, \eqref{Eq:Prop1L} and \eqref{Eq:Prop1a},  
\begin{align*}
\dsq{1,\beta\gamma}^{(1,1)}
=\mult_{2,3}\dsq{1,\beta\otimes\gamma}^{(1,2)} 
=& \mult_{2,3}\left(2 \dsq{1,\beta}^{(1,1)} \otimes \gamma 
- \tau_{(23)}^{1,1} \dsq{1\otimes \gamma,\beta}^{(2,1)} \right) \\
=&2 \dsq{1,\beta}^{(1,1)} (1\otimes \gamma) 
- \mult_{2,3}\left(\tau_{(23)}^{1,1} 1\otimes\dsq{\gamma,\beta}^{(1,1)}\right) \\
=&2 \dsq{1,\beta}^{(1,1)} (1\otimes \gamma) 
+\mult_{2,3}\,( 1\otimes \dsq{\beta, \gamma}^{(1,1)}) \,.
\end{align*} 
Summing both expressions, we get $\dsq{1,-}^{(1,1)}\in \A\otimes \Der(\A)$. 
If we further apply the multiplication map, 
we get the following by \eqref{Eq:Prop1unit}
\begin{equation*}
  2\, D(\beta\gamma) =  2\mult \dsq{1,\beta\gamma}^{(1,1)}
= 2\,\beta D(\gamma) + 2\, D(\beta) \gamma.
\end{equation*}
Hence $D\in \Der(\A)$. 
Next, we show that the extension of $D$ to $\A^{\otimes r}$ given by \eqref{Eq:Dext} satisfies 
\begin{equation} \label{Eq:DalphaR}
     D(\alpha)=\mult_{(1,2)}\dsq{1,\alpha}^{(1,r)}\,, \qquad \alpha \in \A^{\otimes r}.
\end{equation}
This is proved by induction, and we have already established the case $r=1$. For $r\geq 2$, write $\alpha=\alpha'\otimes \alpha''$ where $\alpha'\in \A$, $\alpha'' \in \A^{\otimes (r-1)}$. We need to check that 
\[
\mult_{1,2}\dsq{1,\alpha}^{(1,r)} 
= (\mult_{1,2}\dsq{1,\alpha'}^{(1,1)})\otimes \alpha'' + \alpha'\otimes (\mult_{1,2}\dsq{1,\alpha''}^{(1,r-1)} )
\]
since the right-hand side is $D(\alpha')\otimes \alpha'' + \alpha' \otimes D(\alpha'')$ by induction. 
Performing manipulations similar to those at the beginning of the proof, we can show 
\begin{align*}
    \dsq{1,\alpha'\otimes \alpha''}^{(1,r)}&= 2\, \tau_{(12)}^{1,1} (\alpha'\otimes \dsq{1,\alpha''}^{(1,r-1)} )
    - 1 \otimes \dsq{\alpha',\alpha''}^{(1,r-1)}\,, \\
    \dsq{1,\alpha'\otimes \alpha''}^{(1,r)}&= 2\, \dsq{1,\alpha'}^{(1,1)} \otimes \alpha'' 
    + 1 \otimes \dsq{\alpha',\alpha''}^{(1,r-1)}\,.
\end{align*}
Taking the sum of these two expressions and applying $\mult_{1,2}$ yields 
\begin{align*}
\mult_{1,2}\dsq{1,\alpha}^{(1,r)} 
&=\mult_{1,2}\left(\tau_{(12)}^{1,1} (\alpha'\otimes \dsq{1,\alpha''}^{(1,r-1)} ) + \dsq{1,\alpha'}^{(1,1)} \otimes \alpha^+ \right) \\
&= (\mult_{1,2}\dsq{1,\alpha'}^{(1,1)})\otimes \alpha'' + \alpha'\otimes (\mult_{1,2}\dsq{1,\alpha''}^{(1,r-1)} ) 
\end{align*}
where in the second equality we use \eqref{Eq:Prop1unit}. This is the desired identity. 

Finally, it remains to verify \eqref{Eq:dsq-rs}. 
We already have the following as a consequence of \eqref{Eq:DalphaR} and \eqref{Eq:Prop1unit}: 
\begin{equation} \label{Eq:DalphaRbis}
   \dsq{1,\beta}^{(1,s)} = 1\otimes  D(\beta)\,, \qquad \beta \in \A^{\otimes s}.
\end{equation}
We have on the one hand,
\begin{align*}
\dsq{\alpha,\beta}^{(r,s)}&=\mult_{1,2}\dsq{1\otimes \alpha,\beta}^{(r+1,s)} \\
&=\mult_{1,2}\circ \tau_{(23)}^{s,r}\left(2 \, \dsq{1,\beta}^{(1,s)} \otimes \alpha 
- \dsq{1,\beta \otimes \alpha}^{(1,r+s)}\right) \\
&=2 \tau_{(12)}^{s,r} (\mult_{1,2}\dsq{1,\beta}^{(1,s)} \otimes \alpha ) 
- \tau_{(12)}^{s,r} (\mult_{1,2} \dsq{1,\beta \otimes \alpha}^{(1,r+s)}) \\
&= 2  \alpha \otimes D(\beta) - \mult_{r+s,r+s+1}\circ\tau_{(13)}^{1,s,r}  (\dsq{1,\beta \otimes \alpha}^{(1,r+s)})\,.
\end{align*}
On the other hand, we can write 
\begin{align*}
\dsq{\alpha,\beta}^{(r,s)}&=\mult_{r+s,r+s+1}\dsq{\alpha,\beta\otimes 1}^{(r,s+1)} \\
&=\mult_{r+s,r+s+1}\circ \tau_{(12)}^{s,r}\left(2 \, \beta \otimes \dsq{\alpha,1}^{(r,1)}  
- \dsq{\beta \otimes \alpha,1}^{(r+s,1)}\right) \\
&=\mult_{r+s,r+s+1}\circ \tau_{(12)}^{s,r}\left(-2 \, \tau_{(23)}^{1,r}(\beta \otimes \dsq{1,\alpha}^{(1,r)})  
+ \tau_{(132)}^{1,s,r} \dsq{1,\beta \otimes \alpha}^{(1,r+s)}\right) \\
&=-2\mult_{r+s,r+s+1}\circ \tau_{(123)}^{s,1,r}\left(\, \beta \otimes \dsq{1,\alpha}^{(1,r)}\right)  
+\mult_{r+s,r+s+1}\circ \tau_{(13)}^{1,s,r} \left(\dsq{1,\beta \otimes \alpha}^{(1,r+s)}\right) \\
&=-2\, D(\alpha)\otimes \beta 
+\mult_{r+s,r+s+1}\circ \tau_{(13)}^{1,s,r} \left(\dsq{1,\beta \otimes \alpha}^{(1,r+s)}\right) \,.
\end{align*}
Summing both expressions yields \eqref{Eq:dsq-rs}. 
\end{proof}

\begin{theorem} \label{Thm:RepIdAdapt}
   Assume that $(\A,\dsq{-,-})$ is an $\id$-adapted double transposed Poisson algebra. 
Then, for any $N\geq 1$, the $\kk$-bilinear operation $[-,-]_N:\A_N\times \A_N\to \A_N$ uniquely determined by \eqref{Eq:concis} 
together with the commutative associative multiplication turn $\A_N$ into a transposed Poisson algebra. 
\end{theorem}
 
\begin{proof}
    Assuming that the operation $[-,-]_N$ is well-defined, we have seen in Remark \ref{Rem:pr1} that it is a Lie bracket satisfying the defining condition \eqref{Eq:trPA}. 
    Thus, it suffices to check that this operation is compatible with the relations \eqref{Eq:AN-def} in $\A_N$.
So let us take $\alpha,\alpha',\beta,\gamma\in \mbf T \A$ homogeneous of respective degrees $r,r,s,t$ and 
corresponding pairs of multi-indices $(\mbf i,\mbf j),(\mbf k,\mbf l),(\mbf u,\mbf v)$ of length $r,s,t$. 
Firstly, given $\lambda,\mu \in \kk$, we have 
\begin{align}
   [(\lambda \alpha+\mu \alpha')_{\mbf i\mbf j},\beta_{\mbf k\mbf l}]_N 
=  \dsq{\lambda \alpha + \mu \alpha' , \beta }^{r,s}_{\mbf i\sqcup \mbf k,\mbf j\sqcup \mbf l} 
&= \lambda  \dsq{\alpha, \beta }^{r,s}_{\mbf i\sqcup \mbf k,\mbf j\sqcup \mbf l} 
+ \mu \dsq{\alpha' , \beta }^{r,s}_{\mbf i\sqcup \mbf k,\mbf j\sqcup \mbf l} \\
&=\lambda [ \alpha_{\mbf i\mbf j},\beta_{\mbf k,\mbf l}]_N 
+\mu [ \alpha'_{\mbf i\mbf j},\beta_{\mbf k,\mbf l}]_N 
= [\lambda \alpha_{\mbf i\mbf j}+\mu \alpha'_{\mbf i\mbf j},\beta_{\mbf k \mbf l}]_N 
\end{align}
so the operation $[-,-]_N$ is compatible with the first relation in \eqref{Eq:AN-def} in its left argument. By skewsymmetry, this is also true for the right argument.

Secondly, one has for $1\leq p,q\leq N$ by \eqref{Eq:Prop1L},  
\begin{align}
   [\alpha_{\mbf i\mbf j} 1_{pq},\beta_{\mbf k\mbf l}]_N 
=  \dsq{1\otimes_r \alpha, \beta }^{r+1,s}_{\mbf i\sqcup p \sqcup \mbf k,\mbf j\sqcup q \sqcup \mbf l} 
&= 1_{pq} \, \dsq{\alpha, \beta }^{r,s}_{\mbf i\sqcup \mbf k,\mbf j\sqcup \mbf l}  
= \delta_{pq} \, [\alpha_{\mbf i\mbf j},\beta_{\mbf k\mbf l}]_N  
= [ \delta_{pq} \, \alpha_{\mbf i\mbf j},\beta_{\mbf k\mbf l}]_N  \,,
\end{align}
and this result also holds when $\alpha_{\mbf i\mbf j}$ is removed due to \eqref{Eq:Prop1unit}. 
Thus, the operation $[-,-]_N$ is compatible with the second relation in \eqref{Eq:AN-def} in its left argument, hence in both arguments. 

Thirdly, picking multi-indices  $(\mbf i',\mbf j')$ of length $r-1$, we compute  for any $1\leq u  \leq r-1$ 
by \eqref{Eq:Prop1L} and \eqref{Eq:Prop1M}
\begin{align}
 [(\mult_{u,u+1}\alpha)_{\mbf i' \mbf j'} ,\beta_{\mbf k\mbf l}]_N 
&=\dsq{ \mult_{u,u+1}\alpha,\beta }^{(r-1,s)}_{\mbf i'\sqcup \mbf k,\mbf j'\sqcup \mbf l}   
=(\mult_{u,u+1}\dsq{ \alpha,\beta }^{(r,s)})_{\mbf i'\sqcup \mbf k,\mbf j'\sqcup \mbf l}  \,.
\end{align} 
Using the multi-indices \eqref{Eq:iojo} of length $r$ depending on $u$ and $o\in \{1,\ldots,N\}$, we get thanks to \eqref{Eq:iojo2}
\[
(\mult_{u,u+1}\dsq{ \alpha,\beta }^{(r,s)})_{\mbf i'\sqcup \mbf k,\mbf j'\sqcup \mbf l} 
= \sum_{1\leq o\leq N} \dsq{ \alpha,\beta }^{(r,s)}_{\mbf i'_o\sqcup \mbf k\,,\,\mbf j'_o\sqcup \mbf l}
= \sum_{1\leq o\leq N} [\alpha_{\mbf i'_o \mbf j'_o} ,\beta_{\mbf k\mbf l}]_N 
=  \Big[\sum_{1\leq o\leq N} \alpha_{\mbf i'_o \mbf j'_o} ,\beta_{\mbf k\mbf l}\Big]_N \,.
\]
Thus, the operation $[-,-]_N$ is compatible with the last relation in \eqref{Eq:AN-def} (in the general form \eqref{Eq:iojo2}) in its left argument, hence it is compatible in both arguments.
\end{proof}

Since $\A_N$ is unital, we know that the Lie bracket given by Theorem~\ref{Thm:RepIdAdapt} is defined by the derivation $[1,-]_N$, cf. Proposition~\ref{Pr:tPA-unit}. In fact, it is determined for each $N$ by the exact same derivation on $\A$ after uniquely extending it to a derivation on $\A_N$ satisfying \eqref{Eq:DerInduce}.  

\begin{corollary}
Let $\A_N$ be equipped with a transposed Poisson algebra structure obtained as in Theorem~\ref{Thm:RepIdAdapt}. 
Then its Lie bracket is defined according to \eqref{Eq:baiDer} by $D_N\in \Der(\A_N)$ which is the lift of the unique derivation $D\in \Der(\A)$ determined by \eqref{Eq:Dalpha}.
\end{corollary}
\begin{proof}
The transposed Poisson bracket comes from an $\id$-adapted double structure on $\A$ by assumption. 
Thus, Proposition \ref{Pr:DerIdAdapt} entails that the operations $\dsq{-,-}^{(r,s)}$ on $\A$ are of the form  \eqref{Eq:dsq-rs}  
obtained (by extension) from the derivation  $D\in \Der(\A)$ satisfying \eqref{Eq:Dalpha}. 
Combining these observations with \eqref{Eq:concis}, we can write 
\begin{equation} 
  [\alpha_{\mbf i\mbf j},\beta_{\mbf k\mbf l}]_N 
=\llbracket \alpha , \beta\rrbracket^{(r,s)}_{\mbf i\sqcup \mbf k,\mbf j\sqcup \mbf l}
=\alpha_{\mbf i\mbf j} \, D(\beta)_{\mbf k\mbf l} 
- D(\alpha)_{\mbf i\mbf j} \, \beta_{\mbf k\mbf l}
=\alpha_{\mbf i\mbf j} \, D_N(\beta_{\mbf k\mbf l}) 
- D_N(\alpha_{\mbf i\mbf j}) \, \beta_{\mbf k\mbf l}\,,
\end{equation}
for pairs of multi-indices $(\mbf i,\mbf j),(\mbf k,\mbf l)$ of length $r,s$.
Hence the Lie bracket is defined from  $D_N\in \Der(\A_N)$ naturally induced by $D\in \Der(\A)$. 
\end{proof}

\begin{remark}
The induced derivation $D_N\in \Der(\A_N)$ is $\GLN$-equivariant for the action \eqref{Eq:Act} by construction, thus the transposed Poisson algebra structure on $\A_N$ obtained in Theorem~\ref{Thm:RepIdAdapt}
is also $\GLN$-equivariant. By Proposition \ref{Pr:tPA-Red}, this structure can be restricted to $\A_N^{\GLN}$.  
These observations will be generalized in Subsection~\ref{ss:dtPA-Inv}. 
\end{remark}

%%%%%%%%%% NEW SECTION %%%%%%%%%
%%%%%%%%%% NEW SECTION %%%%%%%%%
%%%%%%%%%% NEW SECTION %%%%%%%%%
%%%%%%%%%% NEW SECTION %%%%%%%%%
%%%%%%%%%% NEW SECTION %%%%%%%%%
%%%%%%%%%% NEW SECTION %%%%%%%%%
\section{General definition} \label{Sec:Gen}

In this section, we start by gathering the notions from \cite{safonkin2025double} needed to define double transposed Poisson algebras (cf. Definition \ref{Def:trDPA}), whose structures are subsequently studied. In particular, we shall see how to induce transposed Poisson algebra structures on $N$-th representation algebras and their $\Gl_N(\kk)$-invariant subrings. 

\subsection{Preliminaries} \label{ss:Prelim}

We embed the symmetric group $S_n$ in $S_{n+1}$ as its subgroup fixing the last node. 
More generally, for any $n, m\in \mathbb{Z}_{\geq 1}$, we consider the homomorphism of \textit{concatenation of permutations}
\begin{align*}
    S_n\times S_{m}\hookrightarrow S_{n+m}, \quad (w,\tau)\mapsto w\times \tau, \quad \text{where }
    (w\times \tau)(j)= 
\left\{ 
\begin{array}{ll}
     w(j)& 1\leq j \leq n ,  \\
     \tau(j-n)+n& n+1\leq j \leq n+m  .
\end{array}
\right.
\end{align*}
Given $\tau\in S_m$ and $k_1,\ldots,k_m\in\mathbb{Z}_{\geq 0}$, we denote by $\tau^{k_1,\ldots,k_m}\in S_{k_1+\ldots+k_m}$ the permutation that splits the $k_1+\ldots+k_m$ nodes in $m$ blocks of respective lengths $k_1,\ldots,k_m$, and permutes these blocks according to $\tau$. 
For example, $(12)^{n,m}(j)=m+j$ if $1\leq j \leq n$ and $(12)^{n,m}(j)=j-n$ if $n+1\leq j \leq n+m$. 

We also introduce the operation of \textit{canonical projection} \cite{kerov2004harmonic} 
\begin{align*}
    S_{n+1}\longrightarrow S_n, \quad u\mapsto u_\ast, \quad \text{where }
    u_\ast(j)= 
\left\{ 
\begin{array}{ll}
     u(j+1)-1& \text{if } u(j+1)\neq 1 ,  \\
     u^2(j+1)-1&\text{if } u(j+1)=1 .
\end{array}
\right.
\end{align*} 
If $u=\operatorname{id}_1\times w\in S_{n+1}$, we simply get $u_*=w\in S_n$. 
If $u(1)\neq 1$, this operation amounts to `jumping over' the first node by sending $u^{-1}(1)$ to $u(1)$ before relabelling the nodes $\{2,\ldots,n+1\}$ as $\{1,\ldots,n\}$.  
This map is motivated by the following result. 
\begin{proposition}[\cite{kerov2004harmonic}]\label{prop9}
Consider the two-sided action of $S_n$ on $S_{n+1}$ via $S_n\simeq\operatorname{id}_1\times S_n\subset S_{n+1}$.
Then, the map $S_{n+1}\longrightarrow S_n$, $u\mapsto u_*$ is a (two-sided) $S_n$-equivariant map. 
Moreover, if $n\geq 4$, this is the only map $S_{n+1}\longrightarrow S_n$ with this property.
\end{proposition}

\begin{definition}
    A \textit{diagonal $\SSS$-bimodule} $M$ is a graded vector space $M=\bigoplus_{n\geq 0}M^{(n)}$, where each $M^{(n)}$ is an $S_n$-bimodule. The degree of a homogeneous element $m\in M^{(n)}$ is denoted by $|m|=n$. A \textit{morphism of diagonal $\SSS$-bimodules} $f:M\longrightarrow N$ is a collection of $S_n$-bimodule homomorphisms $f_n:M^{(n)}\longrightarrow N^{(n)}$. The adjoint $\SSS$-module of $M$ is the same graded vector space $M$ endowed with the adjoint action by 
\[
\operatorname{Ad}(\sigma)(m):=\sigma\cdot m\cdot\sigma^{-1}, \qquad \sigma\in S_n,\,\,m\in M^{(n)}.
\]
\end{definition}
Diagonal $\SSS$-bimodules form a symmetric monoidal category with the tensor product $\stimes$ is defined as follows: given diagonal $\SSS$-bimodules $M$ and $N$, the component of degree $n\geq 0$ of $M\stimes N$ is given by 
\begin{align}
%&M\stimes N:=\bigoplus\limits_{n\geq 0}(M\stimes N)^{(n)} \,, \qquad \text{where } \\
 &   (M\stimes N)^{(n)}=\bigoplus\limits_{i+j=n}\Bbbk[S_n]\otimes_{S_{i}\times S_{j}}\big(M^{(i)}\otimes N^{(j)}\big)\otimes_{S_{i}\times S_{j}}\Bbbk[S_n].
\end{align}
Here, $S_{i}\times S_{j}$ is regarded as a subgroup of $S_n$ via concatenation of permutations.

\medskip 

Let $\A$ be a unital associative algebra over $\Bbbk$. 
We form the vector space $\A_{\natural}:=\A/[\A,\A]$ and write 
the associated linear map as $\A\to \A_\natural$, $a\mapsto \overline{a}$. 
We set
\begin{align}
      \O(\A):=\SSym(\A_{\natural}\oplus \A[-1])=\bigoplus\limits_{n\geq 0}\A^{\otimes n}\otimes \sss\otimes \Bbbk[S_n].
\end{align}
The graded vector space $\O(\A)$ is naturally an $\SSS$-bimodule and an associative algebra for 
\begin{align*}
 \sigma\cdot \big(a\otimes f\otimes u\big)\cdot \tau:=&\,\sigma(a)\otimes f\otimes \sigma\cdot u\cdot \tau\,, \\
 (a\otimes f\otimes u)(b\otimes g\otimes v):=& (a\otimes b)\otimes (f\cdot g)\otimes u\times v, 
\end{align*}
where  $a\in\A^{\otimes n}$, $b\in\A^{\otimes m}$, $f,g\in\sss$, $u,\sigma,\tau\in S_n$, $v\in S_m$. 
This turns $\O(\A)$ into a commutative algebra in the category of $\SSS$-modules with respect to the adjoint $\SSS$-module structure.
Denoting by $\mathds{1}$ the unit of $\sss$, 
$\O(\A)$ is unital for 
\begin{equation} \label{Eq:Unit-OA}
    \mathds{1}_{\O(\A)}:= \mathds{1} \otimes \id_0 \in\O(\A)^{(0)} . 
\end{equation}

The algebra $\O(\A)$ first appeared in \cite{ginzburg2010differentialoperatorsbvstructures} under the name \textit{Fock algebra}. In \cite{safonkin2025double}, it was called \textit{double coordinate ring} because it naturally replaces $\A_N$ (cf. Subsection \ref{ss:MotivNot}) in noncommutative constructions compatible with the Kontsevich-Rosenberg principle.

    Introduce a graded linear map $\pi:\O(\A)\longrightarrow \O(\A)$ of degree $-1$ by $\pi:\O(\A)^{(n)}\longrightarrow \O(\A)^{(n-1)}$
\begin{equation}
    \pi(a\otimes f\otimes u)=\begin{cases}
        (a_2\otimes\ldots\otimes a_n)\otimes \overline{a_1}\cdot f\otimes u_*, &\text{if}\ u(1)=1,\\[5pt]
        m_{1,u(1)}(a) \otimes f\otimes u_*, &\text{if}\ u(1)>1,
    \end{cases}
\end{equation}

where 
\begin{itemize}
    \item $u\mapsto u_*$ is the canonical projection;
    
    \item $a=a_1\otimes\ldots\otimes a_n\in \A^{\otimes n}$, $f\in\sss$, and $u\in S_n$;
    
    \item the map $m_{1,k}:\A^{\otimes n}\rightarrow \A^{\otimes n-1}$ multiplies the $k$-th tensor factors by the first one on the left, cf. \eqref{Eq:mult-uv}.
\end{itemize}

For instance, $\pi(a\otimes f\otimes \operatorname{id}_1)=1\otimes\overline{a}\cdot f\otimes\operatorname{id}_0$ and $\pi(1\otimes f\otimes\operatorname{id}_0)=0$ for $a\in \A$, $f\in\sss$.

Denote by $\widehat{1}:\O(\A)\longrightarrow\O(\A)$ the multiplication operator by $1\otimes\mathds{1}\otimes\operatorname{id}_1\in\O(\A)^{(1)}$ on the left, where $1\in \A$, $\mathds{1}\in\sss$, and $\id_1\in S_1$ are the units. Explicitly, one has
\begin{equation}
    \widehat{1}(\alpha)=(1\otimes a)\otimes f\otimes (\operatorname{id}_1\times u)\in\O(\A)^{(n+1)},    
\end{equation}
where $\alpha=a\otimes f\otimes u\in\O(\A)^{(n)}$. 

We will need to produce elements of $\A_N$ out of $\O(\A)$. Assume that two tuples of indices ranging from $1$ to $N$ are given, say $\mathbf{i}=(i_1,\ldots,i_n)$ and $\mathbf{j}=(j_1,\ldots,j_n)$. Take any $a_1,\ldots, a_n\in \A$, a permutation $u\in S_n$, and any $f_1,\ldots, f_m\in \A_{\natural}$. Then the element $\alpha=(a_1\otimes\ldots\otimes a_n) \otimes f_1\cdot\ldots\cdot f_m\otimes u$ belongs to $\O(\A)^{(n)}$ and the element of the coordinate ring $\A_N$ corresponding to $\mathbf{i}$, $\mathbf{j}$, and $\alpha$, denoted by $\alpha_{\mathbf{i}\mathbf{j}}$, is by definition
\begin{equation}\label{f28}
    \alpha_{\mathbf{i}\mathbf{j}}=(a_1)_{i_{u^{-1}(1)}j_1}\ldots (a_n)_{i_{u^{-1}(n)}j_n}\tr(f_1)\ldots\tr(f_m)\in\A_N.
\end{equation}

Here $\tr(f)=\sum_{j=1}^N F_{jj}$ for $F\in \A$ such that $\overline{F}=f\in \A_\natural$ (this does not depend on the chosen lift $F$ of $f$).
When $m=1$, $f=\mathds{1}$ (resp. $f=\overline{1}$) and $u=\id_{n}$, \eqref{f28} is exactly (resp. is $N$ times) the expression \eqref{Eq:spanAN}.
   
Extend this definition linearly to arbitrary elements of $\Mat^{\otimes n}$: if $X=\sum\limits_{\mathbf{i},\mathbf{j}\in[1,N]^n} X_{\mathbf{i}\mathbf{j}}E^*_{\mathbf{i}\mathbf{j}}$, where $E^*_{\mathbf{i}\mathbf{j}}:=E^*_{i_1j_1}\otimes\ldots\otimes E^*_{i_nj_n}$ and $X_{\mathbf{i}\mathbf{j}}\in \kk$, then for any $\alpha\in\O(\A)^{(n)}$ we set
\begin{align}
    \left(\alpha\middle|X\right):=\sum\limits_{\mathbf{i},\mathbf{j}\in[1,N]^n} X_{\mathbf{i}\mathbf{j}}\, \alpha_{\mathbf{i}\mathbf{j}} \, \in \A_N.
\end{align}

\subsection{Transposed Poisson and double transposed Poisson algebras}

Fix  an associative algebra $\A$. 
An $\SSS$-bimodule map $\llbracket-,-\rrbracket:\O(\A)\stimes\O(\A)\longrightarrow \O(\A)$ is a linear map whose restriction 
\[
\O(\A)^{(n)}\otimes \O(\A)^{(m)}\longrightarrow \O(\A)^{(n+m)}, \quad \text{for any } n,m\in \ZZ_{\geq 0}, 
\]
is a morphism of $(S_n\times S_{m})$-bimodules, where the target is endowed with the $(S_n\times S_{m})$-bimodule structure obtained by concatenation of permutations $S_n\times S_{m}\longrightarrow S_{n+m}$. 

\begin{definition}  \label{Def:trDPA}
    A \textit{double transposed Poisson bracket} on $\A$ is an $\SSS$-bimodule map $\llbracket-,-\rrbracket:\O(\A)\stimes\O(\A)\longrightarrow \O(\A)$
    such that for any homogeneous elements $\alpha,\beta,\gamma\in\O(\A)$ one has
    \begin{align}
        \label{f1} &\llbracket\beta,\alpha\rrbracket=-\operatorname{Ad}((12)^{|\alpha|,|\beta|})\llbracket\alpha,\beta\rrbracket,\\[5pt]
        \label{f-jacobi} &\bigl\llbracket \alpha,\llbracket\beta,\gamma\rrbracket\bigr\rrbracket+\operatorname{Ad}((123)^{|\beta|,|\gamma|,|\alpha|}) \bigl\llbracket \beta,\llbracket \gamma,\alpha\rrbracket\bigr\rrbracket +\operatorname{Ad}((132)^{|\gamma|,|\alpha|,|\beta|}) \bigl\llbracket \gamma,\llbracket \alpha,\beta\rrbracket\bigr\rrbracket=0 ,\\[5pt]
        \label{f2} &2\gamma\llbracket\alpha,\beta\rrbracket=\llbracket\gamma \alpha,\beta\rrbracket+\operatorname{Ad}((12)^{|\alpha|,|\gamma|})\llbracket\alpha,\gamma \beta\rrbracket , 
    \end{align}
    and
    \begin{align}
        \label{f3} \llbracket\pi(\alpha),\beta\rrbracket=\pi\llbracket\alpha,\beta\rrbracket,\hspace{30pt} \llbracket\1(\alpha),\beta\rrbracket=\1\llbracket\alpha,\beta\rrbracket.
    \end{align}
We then say that $(\A,\llbracket-,-\rrbracket)$ is a \textit{double transposed Poisson algebra}.
\end{definition}

\begin{remark}
    Any double transposed Poisson bracket on $\A$ is uniquely determined by a collection of maps $\A\otimes \A^{\otimes n}\longrightarrow \A^{\otimes n+1}\otimes \sss\otimes \kk[S_{n+1}]$ due to \eqref{f2} and \eqref{f3}. 
\end{remark}

The following result is a consequence of Proposition 5.8 from \cite{safonkin2025double}, see the definition of double algebras over an operad there (Definition 5.10) and the paragraph below it.

\begin{proposition} \label{Pr:dtPA-induces-tPA-N}
    If $\A$ is a double transposed Poisson algebra, then $\A_N$ is canonically a transposed Poisson algebra with the $\GLN$-equivariant bracket $[-,-]_N$ given by
    \begin{equation} \label{Eq:induced-bracket-AN}
        [(\alpha\vert X),(\beta\vert Y)]_N:=\big(\llbracket\alpha,\beta\rrbracket\vert X\otimes Y\big).
    \end{equation}
\end{proposition}
\begin{proof}
    In the operadic language,
    Definition~\ref{Def:trDPA} states exactly that a double transposed Poisson
    algebra is an operad homomorphism $\mathcal{P}\to\mathcal{E}nd_\A$ from the
    operad $\mathcal{P}$ of transposed Poisson algebras to
    \[
        \mathcal{E}nd_\A:=\Big(\operatorname{Hom}^{\mathrm{adm}}_{\SSS}\big(\O(\A)^{\otimes_{\SSS} n},\O(\A)\big)\Big)_{n\geq 0},
    \]
    where $\operatorname{Hom}^{\mathrm{adm}}_{\SSS}$ denotes the set of admissible
    $\SSS$-bimodule homomorphisms, see \cite[Def.~5.5]{safonkin2025double}.
    Composing it with the natural operad homomorphism
    $\mathcal{E}nd_\A\to\mathcal{E}nd^{\mathrm{red}}_\A$ and with the homomorphism
    from the second item of \cite[Prop.~5.8]{safonkin2025double} yields the claim.

    A more down-to-earth explanation is the following. By \cite[Cor.~5.2]{safonkin2025double}, the formula~\eqref{Eq:induced-bracket-AN}
    defines a well-defined $\GLN$-equivariant linear map $[-,-]_N$ on $\A_N$ for
    every $N\geq 1$, since $\llbracket-,-\rrbracket$ is an $\SSS$-bimodule map
    compatible with $\pi$ and $\1$ by~\eqref{f3}. Under the pairing $(-\vert-)$ the
    adjoint action $\operatorname{Ad}$ becomes the permutation of the matrix
    factors and the product of $\O(\A)$ becomes the product of $\A_N$; hence the
    skewsymmetry, the Jacobi identity, and the transposed Leibniz rule of
    $[-,-]_N$ follow immediately from~\eqref{f1}, \eqref{f-jacobi}, and~\eqref{f2},
    respectively.
\end{proof}

Next, we can get the analog of Propositions~\ref{Pr:tPA-unit} and~\ref{Pr:DerIdAdapt} for the double coordinate ring $\O(\A)$: every double transposed Poisson bracket on $\A$ can be defined using a single derivation of $\O(\A)$ which is an $\SSS$-bimodule homomorphism compatible with $\pi$ and $\1$. Recall that $\O(\A)$ has a unit given in \eqref{Eq:Unit-OA}. 

\begin{theorem} \label{Thm:DerDtPA}
   Let $\A$ be an associative algebra equipped with a double transposed Poisson algebra structure $\llbracket-,-\rrbracket$ in the sense of Definition~\ref{Def:trDPA}.
   Then the linear map
\begin{equation} \label{Eq:Dmap}
        D:\O(\A)\longrightarrow\O(\A),\qquad D(\alpha):=\llbracket\mathds{1}_{\O(\A)},\alpha\rrbracket,
\end{equation}
is a degree-zero derivation of $\O(\A)$ which is an $\SSS$-bimodule homomorphism satisfying the compatibility conditions
\begin{equation} \label{Eq:DerCompat}
        D\circ\pi=\pi\circ D,\qquad D\circ\1=\1\circ D.
\end{equation}
The bracket is recovered from $D$ via
\begin{equation} \label{Eq:BracketFromDer}
        \llbracket\alpha,\beta\rrbracket=\alpha\, D(\beta)-D(\alpha)\,\beta ,
\end{equation}
for any $\alpha,\beta\in\O(\A)$.
\end{theorem}

\begin{proof}
   Since $|\mathds{1}_{\O(\A)}|=0$ and $\llbracket-,-\rrbracket$ is $\SSS$-bilinear, $D$ preserves the grading and is an $\SSS$-bimodule homomorphism.

Applying \eqref{f2} with $\alpha=\mathds{1}_{\O(\A)}$ and arbitrary homogeneous $\beta,\gamma\in\O(\A)$, and using $\gamma\cdot\mathds{1}_{\O(\A)}=\gamma$ together with the fact that the permutation $(12)^{0,|\gamma|}$ is the identity, we obtain
\begin{equation} \label{Eq:DerDtPA-pf-1}
        2\gamma\, D(\beta)=\llbracket\gamma,\beta\rrbracket+D(\gamma\beta).
\end{equation}
Apply $\operatorname{Ad}\bigl((12)^{|\gamma|,|\beta|}\bigr)$ to both sides of \eqref{Eq:DerDtPA-pf-1}. The left-hand side becomes $2\,D(\beta)\,\gamma$, since the permutation $(12)^{|\gamma|,|\beta|}$ swaps the two blocks of $\gamma\, D(\beta)$ which are of sizes $|\gamma|$ and $|D(\beta)|=|\beta|$. By the cyclic skewsymmetry~\eqref{f1}, one has $\operatorname{Ad}\bigl((12)^{|\gamma|,|\beta|}\bigr)\llbracket\gamma,\beta\rrbracket=-\llbracket\beta,\gamma\rrbracket$. Finally, since $D$ is an $\SSS$-bimodule homomorphism and the graded commutativity in $\O(\A)$ gives $\operatorname{Ad}\bigl((12)^{|\gamma|,|\beta|}\bigr)(\gamma\beta)=\beta\gamma$, we get $\operatorname{Ad}\bigl((12)^{|\gamma|,|\beta|}\bigr)D(\gamma\beta)=D(\beta\gamma)$. We thus arrive at
\begin{equation} \label{Eq:DerDtPA-pf-2}
        2\,D(\beta)\,\gamma=-\llbracket\beta,\gamma\rrbracket+D(\beta\gamma).
\end{equation}
Exchanging the elements $\gamma\leftrightarrow\beta$ in \eqref{Eq:DerDtPA-pf-1} yields
\begin{equation} \label{Eq:DerDtPA-pf-3}
        2\beta\, D(\gamma)=\llbracket\beta,\gamma\rrbracket+D(\beta\gamma).
\end{equation}
Summing \eqref{Eq:DerDtPA-pf-2} and \eqref{Eq:DerDtPA-pf-3} yields the Leibniz rule
\begin{equation*}
        D(\beta\gamma)=\beta\, D(\gamma)+D(\beta)\,\gamma,
\end{equation*}
so $D$ is a derivation of the algebra $\O(\A)$. Plugging the Leibniz rule into \eqref{Eq:DerDtPA-pf-3} gives the formula \eqref{Eq:BracketFromDer}.

It remains to verify the compatibility conditions \eqref{Eq:DerCompat}. For any homogeneous $\alpha\in\O(\A)$, applying the cyclic skewsymmetry~\eqref{f1} twice (with trivial permutation $(12)^{|\alpha|,0}=\operatorname{id}$, as $|\mathds{1}_{\O(\A)}|=0$) and the first equality in \eqref{f3} once, we get
\begin{equation*}
        D\pi(\alpha)=\llbracket\mathds{1}_{\O(\A)},\pi(\alpha)\rrbracket=-\llbracket\pi(\alpha),\mathds{1}_{\O(\A)}\rrbracket=-\pi\llbracket\alpha,\mathds{1}_{\O(\A)}\rrbracket=\pi\llbracket\mathds{1}_{\O(\A)},\alpha\rrbracket=\pi D(\alpha).
\end{equation*}
The identity $D\1=\1 D$ is established by the same argument using the second equality in \eqref{f3}.
\end{proof}

Theorem~\ref{Thm:DerDtPA} reduces classifying double transposed Poisson algebra structures on $\A$ to classifying derivations of $\O(\A)$ subject to the compatibility conditions. We show next that these derivations are governed by much simpler data: each one is determined by an ordinary derivation of $\A$ valued in $\A\otimes\sss$. To make this precise, embed $\A$ into $\O(\A)^{(1)}$ via the $\kk$-linear injection
\begin{equation} \label{Eq:A-hat-embed}
    \A\hookrightarrow\O(\A)^{(1)},\qquad a\longmapsto \widehat a:=a\otimes\mathds{1}\otimes\id_1,
\end{equation}
and equip $\O(\A)^{(1)}=\A\otimes\sss$ with the $\A$-bimodule structure in which $\A$ acts on the first tensor factor only:
\begin{equation} \label{Eq:A-bimod-on-OA1}
    b\cdot(c\otimes f)\cdot d:=(bcd)\otimes f,\qquad b,c,d\in\A,\ f\in\sss.
\end{equation}
When $\A$ is finitely generated as a $\kk$-algebra, $\Der(\A,\A\otimes\sss)\simeq\Der(\A,\A)\otimes_\kk\sss$.

\begin{proposition} \label{Pr:DerOA-delta}
    There is a bijective correspondence between derivations $D:\O(\A)\longrightarrow\O(\A)$ of the algebra $\O(\A)$ which are $\SSS$-bimodule homomorphisms and which satisfy the compatibility conditions
        \begin{equation} \label{Eq:two-compat}
            D\circ\pi=\pi\circ D,\qquad D\circ\1=\1\circ D,
        \end{equation}
        and derivations $\delta:\A\longrightarrow\A\otimes\sss$ from $\A$ to the $\A$-bimodule $\A\otimes\sss$ (cf. \eqref{Eq:A-bimod-on-OA1}), i.e., $\kk$-linear maps satisfying
        \begin{equation} \label{Eq:delta-Leibniz}
            \delta(ab)=a\cdot\delta(b)+\delta(a)\cdot b,\qquad a,b\in\A.
        \end{equation}
Explicitly, this is defined by the assignment $D\longmapsto \delta_D$, where $\delta_D(a):=D(\widehat a)$ for $a\in\A$.
\end{proposition}

\begin{proof}
    The product in $\O(\A)$ gives $\widehat a\,\widehat b=(a\otimes b)\otimes\mathds{1}\otimes\id_2\in\O(\A)^{(2)}$ for any $a,b\in\A$. The definition of $\pi$ yields
    \begin{equation} \label{Eq:hat-prod-pi}
        \pi\bigl(\widehat a\,\widehat b\cdot(12)\bigr)
        =\pi\bigl((a\otimes b)\otimes\mathds{1}\otimes(12)\bigr)
        =m_{1,2}(a\otimes b)\otimes\mathds{1}\otimes\id_1=ab\otimes\mathds{1}\otimes\id_1=\widehat{ab}.
    \end{equation}
    The same computation, with $\mathds{1}$ replaced by any $h\in\sss$ and $(a,b)$ replaced by any $(c,d)\in\A\times\A$, gives
    \begin{equation} \label{Eq:pi-on-deg-2}
        \pi\bigl((c\otimes d)\otimes h\otimes(12)\bigr)=cd\otimes h\otimes\id_1\,\in\,\A\otimes\sss=\O(\A)^{(1)}.
    \end{equation}

    \medskip

    Apply $D$ to \eqref{Eq:hat-prod-pi}. The compatibility $D\circ\pi=\pi\circ D$, the $\SSS$-equivariance of $D$, together with the Leibniz rule, yield
    \begin{equation*}
        \delta_D(ab)=D(\widehat{ab})=\pi\bigl(D(\widehat a\,\widehat b)\cdot(12)\bigr)
        =\pi\bigl(\bigl(\delta_D(a)\,\widehat b+\widehat a\,\delta_D(b)\bigr)\cdot(12)\bigr).
    \end{equation*}
    Write $\delta_D(a)=\sum_i a'_i\otimes f_i$ and $\delta_D(b)=\sum_j b'_j\otimes g_j$ with $a'_i,b'_j\in\A$ and $f_i,g_j\in\sss$ (the trivial factor $\id_1\in S_1$ is implicit). Computing the products in $\O(\A)$,
    \begin{align*}
        \delta_D(a)\,\widehat b\cdot(12)&=\textstyle\sum_i (a'_i\otimes b)\otimes f_i\otimes(12), \\
        \widehat a\,\delta_D(b)\cdot(12)&=\textstyle\sum_j (a\otimes b'_j)\otimes g_j\otimes(12),
    \end{align*}
    and applying \eqref{Eq:pi-on-deg-2} summand-by-summand gives
    \begin{equation*}
        \delta_D(ab)=\sum_i(a'_ib)\otimes f_i\;+\;\sum_j(ab'_j)\otimes g_j=\delta_D(a)\cdot b+a\cdot\delta_D(b),
    \end{equation*}
    where the last equality is the definition \eqref{Eq:A-bimod-on-OA1} of the $\A$-bimodule structure on $\A\otimes\sss$. This is precisely the Leibniz rule \eqref{Eq:delta-Leibniz} for $\delta_D$.

    We show how to recover the values of $D$ on each homogeneous component $\O(\A)^{(n)}$ from $\delta_D$.

    \begin{itemize}[leftmargin=3ex]
        \item \emph{Restriction to $\O(\A)^{(0)}=\sss$.} For each $a\in\A$, the definition of $\pi$ on $\O(\A)^{(1)}$ gives $\pi(\widehat a)=\mathds{1}\otimes\overline a\otimes\id_0=\overline a\in\A_\natural\subset\sss$. Combining this with $D\circ\pi=\pi\circ D$, we obtain
        \begin{equation} \label{Eq:D-on-Anat-from-delta}
            D(\overline a)=D(\pi(\widehat a))=\pi(\delta_D(a))\qquad\text{for all }a\in\A.
        \end{equation}
        Since $\sss$ is generated as a $\kk$-algebra by the set $\{\overline a:a\in\A\}$ and the restriction $D|_{\sss}$ is a derivation of $\sss$, the values \eqref{Eq:D-on-Anat-from-delta} on these generators determine $D|_{\sss}$ through the Leibniz rule.

        \item \emph{Restriction to $\O(\A)^{(1)}=\A\otimes\sss$.} Every element of $\O(\A)^{(1)}$ is a $\kk$-linear combination of products $f\,\widehat a$ with $f\in\sss$ and $a\in\A$. By the Leibniz rule,
        \begin{equation*}
            D(f\,\widehat a)=D(f)\,\widehat a+f\,\delta_D(a),
        \end{equation*}
        so $D|_{\O(\A)^{(1)}}$ is determined by $D|_{\sss}$ (computed in the previous item) and $\delta_D$.

        \item \emph{Restriction to $\O(\A)^{(n)}$ for $n\geq 2$.} Every element of $\O(\A)^{(n)}$ is a $\kk$-linear combination of elements of the form $f\,\widehat{a_1}\cdots\widehat{a_n}\cdot u$ with $a_1,\ldots,a_n\in\A$, $f\in\sss$, and $u\in S_n$, because
        \begin{equation*}
            f\,\widehat{a_1}\,\widehat{a_2}\cdots\widehat{a_n}\cdot u=(a_1\otimes a_2\otimes\ldots\otimes a_n)\otimes f\otimes u.
        \end{equation*}
        Iterating the Leibniz rule yields
        \begin{equation*}
            D\bigl(f\,\widehat{a_1}\cdots\widehat{a_n}\cdot u\bigr)=D(f)\,\widehat{a_1}\cdots\widehat{a_n}\cdot u+\sum_{k=1}^{n}f\,\widehat{a_1}\cdots\widehat{a_{k-1}}\,\delta_D(a_k)\,\widehat{a_{k+1}}\cdots\widehat{a_n}\cdot u,
        \end{equation*}
        which expresses $D$ on $\O(\A)^{(n)}$ in terms of $D|_{\sss}$ and $\delta_D$.
    \end{itemize}

    Conversely, given a derivation $\delta:\A\to\A\otimes\sss$, we construct $D_\delta:\O(\A)\to\O(\A)$ satisfying $\delta_{D_\delta}=\delta$ and verify that $D_\delta$ has the properties listed in the proposition. Throughout we use the Sweedler-style notation
    \begin{equation} \label{Eq:sweedler-delta}
        \delta(a)=a'\otimes \overline{a''}\in\A\otimes\sss\qquad(a'\in\A,\ \overline{a''}\in\sss),
    \end{equation}
    with summation over an implicit index suppressed. 
    For the bimodule structure \eqref{Eq:A-bimod-on-OA1}, we can write $b\cdot\delta(c)=bc'\otimes \overline{c''}$ and $\delta(c)\cdot b=c'b\otimes \overline{c''}$.

    The $\kk$-linear map $\A\to\sss$, $a\mapsto\overline{a'}\cdot \overline{a''}$, descends to a $\kk$-linear map $\A_\natural=\A/[\A,\A]\to\sss$. Indeed, for $b,c\in\A$, the Leibniz rule \eqref{Eq:delta-Leibniz} and the cyclic identity $\overline{xy}=\overline{yx}$ in $\A_\natural$ (for $x,y\in \A$) give
    \begin{equation*}
        \pi\bigl(\delta(bc)-\delta(cb)\bigr)
        =\pi\bigl(bc'\otimes \overline{c''}+b'c\otimes \overline{b''}-cb'\otimes \overline{b''}-c'b\otimes \overline{c''}\bigr)
        =(\overline{bc'}-\overline{c'b})\cdot \overline{c''}+(\overline{b'c}-\overline{cb'})\cdot \overline{b''}=0.
    \end{equation*}
    This induced map $\A_\natural\to\sss$ extends uniquely to a $\kk$-derivation $\widetilde D_\delta:\sss\to\sss$ characterized by
    \begin{equation} \label{Eq:Dtilde-on-Anat}
        \widetilde D_\delta(\overline a)=\overline{a'}\cdot\overline{a''},\qquad a\in\A.
    \end{equation}

    Define a $\kk$-linear map $D_\delta:\O(\A)\to\O(\A)$ on each homogeneous component $\O(\A)^{(n)}$ by
    \begin{equation}
        \begin{aligned}
           \label{Eq:Ddelta-formula}
        D_\delta\bigl((a_1\otimes\ldots\otimes a_n)\otimes f\otimes u\bigr)
        \;:=&\;\sum_{k=1}^{n}\bigl(a_1\otimes\ldots\otimes a_{k-1}\otimes a_k'\otimes a_{k+1}\otimes\ldots\otimes a_n\bigr)\otimes\bigl(\overline{a_k''}\cdot f\bigr)\otimes u \\
        &+(a_1\otimes\ldots\otimes a_n)\otimes\widetilde D_\delta(f)\otimes u. 
        \end{aligned}
    \end{equation}
    The right-hand side is $\kk$-multilinear in $a_1,\ldots,a_n$ and $\kk$-linear in $f$, so \eqref{Eq:Ddelta-formula} defines a $\kk$-linear map of degree zero. Setting $n=1$, $a_1=a$, $f=\mathds{1}$, $u=\id_1$ and using $\widetilde D_\delta(\mathds{1})=0$ gives
    \begin{equation*}
        D_\delta(\widehat a)=a'\otimes \overline{a''}\otimes\id_1=\delta(a),
    \end{equation*}
    so $\delta_{D_\delta}=\delta$. It remains to verify that $D_\delta$ is an $\SSS$-bimodule derivation of $\O(\A)$ compatible with both $\pi$ and $\1$.

    For $\alpha=(a_1\otimes\ldots\otimes a_n)\otimes f\otimes u$ and $\beta=(b_1\otimes\ldots\otimes b_m)\otimes g\otimes v$ one has
    \begin{equation*}
        \alpha\beta=(a_1\otimes\ldots\otimes a_n\otimes b_1\otimes\ldots\otimes b_m)\otimes(f\cdot g)\otimes(u\times v).
    \end{equation*}
Apply \eqref{Eq:Ddelta-formula} to $\alpha\beta\in\O(\A)^{(n+m)}$. The substitutions at positions $k\in\{1,\ldots,n\}$ act on the $a$-block, and the resulting summands factor as $\bigl(\text{$k$-th summand of $D_\delta(\alpha)$}\bigr)\cdot\beta$; the substitutions at positions $k\in\{n+1,\ldots,n+m\}$ act on the $b$-block and factor as $\alpha\cdot\bigl(\text{$(k-n)$-th summand of $D_\delta(\beta)$}\bigr)$. Finally, the Leibniz rule $\widetilde D_\delta(fg)=\widetilde D_\delta(f)\,g+f\,\widetilde D_\delta(g)$ for the derivation $\widetilde D_\delta$ splits the $\widetilde D_\delta$-summand of $D_\delta(\alpha\beta)$ as $\bigl(\widetilde D_\delta\text{-summand of }D_\delta(\alpha)\bigr)\cdot\beta+\alpha\cdot\bigl(\widetilde D_\delta\text{-summand of }D_\delta(\beta)\bigr)$. Summing these three contributions yields $D_\delta(\alpha\beta)=D_\delta(\alpha)\,\beta+\alpha\,D_\delta(\beta)$. Right $S_n$-equivariance is immediate from \eqref{Eq:Ddelta-formula}: the permutation $u$ appears only in the last tensor slot. For left $S_n$-equivariance, observe that under $\sigma\in S_n$ the entry of $\sigma\cdot\alpha$ at position $\sigma(k)$ equals $a_k$, so the $\sigma(k)$-th summand of $D_\delta(\sigma\cdot\alpha)$ has $a_k'$ at position $\sigma(k)$ and $a_k''$ in the $\sss$-slot; 
%reindexing by $k'=\sigma(k)$, 
this is precisely $\sigma\cdot\bigl(\text{$k$-th summand of }D_\delta(\alpha)\bigr)$.

    Let us prove $D_\delta\circ\pi=\pi\circ D_\delta$. On $\O(\A)^{(0)}=\sss$ both sides vanish, since $\pi$ vanishes in degree zero and $D_\delta$ preserves the grading; we may therefore assume $n\geq 1$ and verify the identity on $\alpha=(a_1\otimes\ldots\otimes a_n)\otimes f\otimes u\in\O(\A)^{(n)}$ by case analysis on $u(1)$.

    \textit{Case $u(1)=1$.} In that case 
    \begin{equation*}
        \pi(\alpha)=(a_2\otimes\ldots\otimes a_n)\otimes(\overline{a_1}\cdot f)\otimes u_*.
    \end{equation*}
Applying \eqref{Eq:Ddelta-formula} with \eqref{Eq:Dtilde-on-Anat} and the Leibniz rule for $\widetilde D_\delta$ (so that $\widetilde D_\delta(\overline{a_1}\cdot f)=\overline{a_1'}\cdot \overline{a_1''}\cdot f+\overline{a_1}\cdot \widetilde D_\delta(f)$) gives
    \begin{align*}
        D_\delta(\pi(\alpha))
        &=\sum_{k=2}^{n}(a_2\otimes\ldots\otimes a_k'\otimes\ldots\otimes a_n)\otimes(\overline{a_k''}\cdot \overline{a_1}\cdot f)\otimes u_* \\
        &\quad+(a_2\otimes\ldots\otimes a_n)\otimes(\overline{a_1'}\cdot \overline{a_1''}\cdot f)\otimes u_* \\
        &\quad+(a_2\otimes\ldots\otimes a_n)\otimes(\overline{a_1}\cdot \widetilde D_\delta(f))\otimes u_*.
    \end{align*}
    Each summand of $D_\delta(\alpha)$ in \eqref{Eq:Ddelta-formula} still has $u(1)=1$, so $\pi$ acts on it via the first case of its definition. The $k=1$ summand contributes the middle term above, the $k\geq 2$ summands contribute the first sum, and the $\widetilde D_\delta(f)$-summand contributes the last term. Thus $\pi(D_\delta(\alpha))=D_\delta(\pi(\alpha))$.

    \textit{Case $u(1)=j>1$.} Then
    \begin{equation*}
        \pi(\alpha)=(a_2\otimes\ldots\otimes a_{j-1}\otimes a_1a_j\otimes a_{j+1}\otimes\ldots\otimes a_n)\otimes f\otimes u_*.
    \end{equation*}
    Apply \eqref{Eq:Ddelta-formula} to $\pi(\alpha)$: at position $j-1$, the substitution uses $\delta(a_1a_j)=a_1a_j'\otimes \overline{a_j''}+a_1'a_j\otimes \overline{a_1''}$ from the Leibniz rule for $\delta$, producing two terms; at each position $k-1$ with $k\in\{2,\ldots,j-1\}\cup\{j+1,\ldots,n\}$, the substitution uses $\delta(a_k)=a_k'\otimes \overline{a_k''}$. Conversely, each of the $n+1$ summands of $D_\delta(\alpha)$ has $u(1)=j>1$, so $\pi$ acts on it via $m_{1,j}$ on the $\A^{\otimes n}$-slot. Computing $\pi(D_\delta(\alpha))$ summand by summand:
    \begin{itemize}
        \item the $k=1$ summand maps to $(a_2\otimes\ldots\otimes a_{j-1}\otimes a_1'a_j\otimes a_{j+1}\otimes\ldots\otimes a_n)\otimes \overline{a_1''}\cdot f\otimes u_*$;
        \item the $k=j$ summand maps to $(a_2\otimes\ldots\otimes a_{j-1}\otimes a_1a_j'\otimes a_{j+1}\otimes\ldots\otimes a_n)\otimes \overline{a_j''}\cdot f\otimes u_*$;
        \item for $k\in\{2,\ldots,j-1,j+1,\ldots,n\}$, the $k$-th summand maps to the element of $\O(\A)^{(n-1)}$ obtained from $m_{1,j}(a_1\otimes\ldots\otimes a_n)$ by replacing $a_k$ at position $k-1$ with $a_k'$, with $\sss$-slot $\overline{a_k''}\cdot f$ and $S_{n-1}$-slot $u_*$;
        \item the $\widetilde D_\delta(f)$-summand maps to $m_{1,j}(a_1\otimes\ldots\otimes a_n)\otimes\widetilde D_\delta(f)\otimes u_*$.
    \end{itemize}
    By the Leibniz expansion $\delta(a_1a_j)=a_1a_j'\otimes \overline{a_j''}+a_1'a_j\otimes \overline{a_1''}$, the first two items form together the substitution at position $j-1$ in $D_\delta(\pi(\alpha))$; for $k\neq 1,j$, the third item is the substitution at position $k-1$ in $D_\delta(\pi(\alpha))$; the fourth is the $\widetilde D_\delta(f)$-summand of $D_\delta(\pi(\alpha))$. Hence $\pi(D_\delta(\alpha))=D_\delta(\pi(\alpha))$ in this case as well.

    Finally, let us prove that $D_\delta\circ\1=\1\circ D_\delta$. The Leibniz rule \eqref{Eq:delta-Leibniz} applied to $1\cdot 1$ gives $\delta(1)=2\delta(1)$, so $\delta(1)=0$. By \eqref{Eq:Ddelta-formula}, $D_\delta(\widehat 1)=\delta(1)=0$, and the Leibniz rule for $D_\delta$ then yields, for any $\alpha\in\O(\A)$,
    \begin{equation*}
        D_\delta(\widehat 1\cdot\alpha)=D_\delta(\widehat 1)\cdot\alpha+\widehat 1\cdot D_\delta(\alpha)=\widehat 1\cdot D_\delta(\alpha),
    \end{equation*}
    i.e., $D_\delta\circ\1=\1\circ D_\delta$.
\end{proof}

\medskip

We now identify the $\id$-adapted double transposed Poisson algebras of Definition~\ref{Def:id-trDPA} as the special case of Definition~\ref{Def:trDPA} for which the underlying derivation $\delta$ from Proposition~\ref{Pr:DerOA-delta} takes values in $\A\otimes\kk\mathds{1}\subset\A\otimes\sss$.

Extend the embedding \eqref{Eq:A-hat-embed} to all tensor powers by setting, for $\alpha=a_1\otimes\ldots\otimes a_n\in\A^{\otimes n}$,
\begin{equation} \label{Eq:hat-tensor}
    \widehat\alpha:=\alpha\otimes\mathds{1}\otimes\id_n\;\in\;\O(\A)^{(n)},\qquad n\geq 0.
\end{equation}
The assignment $\alpha\mapsto\widehat\alpha$ defines an injective homomorphism of associative unital $\kk$-algebras $\mbf T\A\hookrightarrow\O(\A)$, with image
\begin{equation*}
    \widehat{\mbf T\A}:=\bigoplus_{n\geq 0}\widehat{\A^{\otimes n}}=\bigoplus_{n\geq 0}\A^{\otimes n}\otimes\kk\mathds{1}\otimes\kk\id_n\subset\O(\A).
\end{equation*}

\begin{proposition} \label{Pr:dtPA-id-adapted}
    Let $\A$ be an associative unital algebra equipped with a double transposed Poisson algebra structure $\dsq{-,-}$ in the sense of Definition~\ref{Def:trDPA}, and let $\delta:\A\to\A\otimes\sss$ be the associated derivation supplied by Theorem~\ref{Thm:DerDtPA} and Proposition~\ref{Pr:DerOA-delta}. The following conditions are equivalent.
    \begin{enumerate}[label=\upshape(\roman*),leftmargin=3ex]
        \item The bracket restricts to $\widehat{\mbf T\A}$, i.e., $\dsq{\widehat\alpha,\widehat\beta}\in\widehat{\A^{\otimes(r+s)}}$ for all $\alpha\in\A^{\otimes r}$ and $\beta\in\A^{\otimes s}$ with $r,s\geq 0$.
        \item $\delta(\A)\subset\A\otimes\kk\mathds{1}$.
        \item There exists a unique $D\in\Der(\A)$ such that $\delta(a)=D(a)\otimes\mathds{1}$ for all $a\in\A$.
    \end{enumerate}
    When these equivalent conditions hold, by \textup{(i)} the bracket $\dsq{-,-}$ on $\O(\A)$ preserves the subalgebra $\widehat{\mbf T\A}$. Under the identification of $\mbf T\A$ with $\widehat{\mbf T\A}$, this restriction yields a family of $\kk$-bilinear operations
    \begin{equation} \label{Eq:id-from-double}
       \dsq{-,-}^{(r,s)}\colon\A^{\otimes r}\otimes\A^{\otimes s}\longrightarrow\A^{\otimes (r+s)} \qquad (r,s\geq 1),
    \end{equation}
    which is the  $\id$-adapted double transposed Poisson algebra structure on $\A$ attached to $D$ via \eqref{Eq:dsq-rs}; explicitly,
    \begin{equation*}
        \dsq{\alpha,\beta}^{(r,s)}=\alpha\otimes D(\beta)-D(\alpha)\otimes\beta,\qquad \alpha\in\A^{\otimes r},\ \beta\in\A^{\otimes s},
    \end{equation*}
    where on the right $D$ is extended to $\A^{\otimes n}$ via \eqref{Eq:Dext}. Conversely, every  $\id$-adapted double transposed Poisson algebra structure on $\A$ arises in this way from a unique double transposed Poisson algebra structure on $\A$ satisfying the equivalent conditions above.
\end{proposition}

\begin{proof}
    The equivalence \textup{(iii)}$\Leftrightarrow$\textup{(ii)} is clear. 

Next, we start by showing that \textup{(i)} is equivalent to 
\begin{enumerate}
    \item[\textup{(i')}] The derivation $D_\delta$ obtained in Theorem \ref{Thm:DerDtPA} (which is defined by $\delta$ through \eqref{Eq:Ddelta-formula}) preserves $\widehat{\mbf T\A}$.
\end{enumerate}

Recall from Theorem~\ref{Thm:DerDtPA} that $\dsq{x,y}=x\,D_\delta(y)-D_\delta(x)\,y$ for $x,y\in\O(\A)$. If $D_\delta$ preserves $\widehat{\mbf T\A}$ then so does the bracket, since $\widehat{\mbf T\A}$ is a subalgebra. 
Conversely, $\mathds{1}_{\O(\A)}\in\widehat{\mbf T\A}$ and the identity $\dsq{\mathds{1}_{\O(\A)},y}=D_\delta(y)$ shows that $D_\delta(\widehat{\mbf T\A})\subset\widehat{\mbf T\A}$.  

We now show \textup{(i')}$\Leftrightarrow$\textup{(ii)}: specializing \eqref{Eq:Ddelta-formula} with $f=\mathds{1}$, $u=\id_n$ and using $\widetilde D_\delta(\mathds{1})=0$, we obtain (with $\delta(a_k)=:a_k'\otimes \overline{a_k''}$ as in \eqref{Eq:sweedler-delta})
    \begin{equation} \label{Eq:Ddelta-on-id}
       D_\delta(\widehat\alpha)=\sum_{k=1}^{n}(a_1\otimes\ldots\otimes a_k'\otimes\ldots\otimes a_n)\otimes \overline{a_k''}\otimes\id_n
    \end{equation}
    for $\alpha=a_1\otimes\ldots\otimes a_n$. Setting $n=1$ gives $D_\delta(\widehat a)=\delta(a)\otimes\id_1$, so \textup{(i')} forces $\delta(a)\in\A\otimes\kk\mathds{1}$ for all $a\in\A$. Conversely, if \textup{(ii)} holds then every $\overline{a_k''}$ in \eqref{Eq:Ddelta-on-id} lies in $\kk\mathds{1}$, and $D_\delta$ preserves $\widehat{\mbf T\A}$.

    Suppose now that the equivalent conditions hold and write $\delta(a)=D(a)\otimes\mathds{1}$ as in \textup{(iii)}. Then \eqref{Eq:Ddelta-on-id} reads $D_\delta(\widehat\alpha)=\widehat{D(\alpha)}$ for every $\alpha\in\A^{\otimes n}$, where on the right $D$ acts on $\A^{\otimes n}$ via \eqref{Eq:Dext}. Since $\widehat{(-)}$ is an algebra map, the bracket on $\widehat{\mbf T\A}$ takes the form
    \begin{equation*}
       \dsq{\widehat\alpha,\widehat\beta}=\widehat\alpha\,\widehat{D(\beta)}-\widehat{D(\alpha)}\,\widehat\beta=\widehat{\big(\alpha\otimes D(\beta)-D(\alpha)\otimes\beta\big)},
    \end{equation*}
    which under \eqref{Eq:hat-tensor} is precisely \eqref{Eq:dsq-rs}. The resulting family $\dsq{-,-}^{(r,s)}$ is an $\id$-adapted double transposed Poisson algebra structure on $\A$ by Proposition~\ref{Prop1}. 

    Conversely, given an $\id$-adapted structure on $\A$, Proposition~\ref{Pr:DerIdAdapt} extracts a unique $D\in\Der(\A)$ with $\dsq{\alpha,\beta}^{(r,s)}=\alpha\otimes D(\beta)-D(\alpha)\otimes\beta$. Setting $\delta(a):=D(a)\otimes\mathds{1}$ defines a derivation $\delta:\A\to\A\otimes\sss$ to which Proposition~\ref{Pr:DerOA-delta} associates a derivation $D_\delta$ of $\O(\A)$; the calculation $\dsq{\widehat\alpha,\widehat\beta}=\widehat{\alpha\otimes D(\beta)-D(\alpha)\otimes\beta}$ established above shows that the restriction of this bracket to $\widehat{\mbf T\A}$ recovers the given $\id$-adapted brackets. Uniqueness then follows by a chain of forced choices: Proposition~\ref{Pr:DerIdAdapt} extracts $D$ uniquely from the $\id$-adapted bracket, condition~\textup{(iii)} forces $\delta=D\otimes\mathds{1}$, and Proposition~\ref{Pr:DerOA-delta} recovers the double transposed Poisson bracket on $\A$ uniquely from $\delta$.
\end{proof}

\subsection{\texorpdfstring{$H_0$}{H0}-transposed Poisson structures} \label{Sec:0-degree}

Let $\A$ be a finitely generated associative $\kk$-algebra. The restriction of a double transposed Poisson bracket $\dsq{-,-}$ on $\A$ (Definition~\ref{Def:trDPA}) to the degree-zero component $\sss=\O(\A)^{(0)}$ produces a bracket $\{-,-\}$ on $\sss$ that satisfies the transposed Poisson axioms together with a further constraint coming from the $\pi$-compatibility at mixed degree, in the spirit of Crawley-Boevey's $H_0$-Poisson structures \cite{crawley-boevey2011poisson}. We package this idea into the following definition.

\begin{definition} \label{Def:trH0P}
    An \emph{$H_0$-transposed Poisson structure} on an associative $\kk$-algebra $\A$ is a transposed Poisson algebra structure $\{-,-\}$ on $\sss$ such that for every $a\in\A$ there exists $d_a\in(\Der(\A)\oplus\kk\cdot\mathrm{id}_\A)\otimes \sss$ with
    \begin{equation} \label{Eq:trH0P}
        \{\overline a,\overline b\}=\pi\bigl(d_a(b)\bigr)\qquad\text{for all }b\in\A,
    \end{equation}
    where $\pi\colon\A\otimes\sss\to\sss$, $x\otimes f\mapsto\overline x\cdot f$.
\end{definition}

Note that \eqref{Eq:trH0P} only depends on the class $\overline{b} \in \A_\natural$ of $b$. 
Equation~\eqref{Eq:trH0P} is the natural transposed analog of Crawley-Boevey's defining property for $H_0$-Poisson structures \cite{crawley-boevey2011poisson}.

\begin{proposition} \label{Thm:0-deg-is-H0P}
    The restriction to $\sss\otimes\sss$ of any double transposed Poisson algebra structure $\dsq{-,-}$ on $\A$ (Definition~\ref{Def:trDPA}) is an $H_0$-transposed Poisson structure on $\A$ with the associated $d_a$ given by
    \begin{equation} \label{Eq:trH0P-witness-dtPA}
        d_a:=\delta\cdot\overline a\;-\;\mathrm{id}_\A\otimes\pi(\delta(a))\;\in\;(\Der(\A)\oplus\kk\cdot\mathrm{id}_\A)\otimes_\kk\sss,
    \end{equation}
    where $\delta\colon\A\to\A\otimes\sss$ is the derivation of Proposition~\ref{Pr:DerOA-delta} and $\delta\cdot\overline a$ denotes $\delta$ post-multiplied on the $\sss$-factor by $\overline a$. 
\end{proposition}
\begin{proof}
    The $S_n$-Ad-twists in Definition~\ref{Def:trDPA} act trivially on tensor products of degree-zero elements, so the cyclic skewsymmetry~\eqref{f1}, the Jacobi identity, and the compatibility~\eqref{f2} reduce, for $f,g,h\in\sss$, to
    \begin{align*}
        &\{g,f\}=-\{f,g\},\\
        &\{f,\{g,h\}\}+\{g,\{h,f\}\}+\{h,\{f,g\}\}=0,\\
        &2h\{f,g\}=\{hf,g\}+\{f,hg\},
    \end{align*}
    namely the skewsymmetry, Jacobi identity, and transposed Poisson compatibility~\eqref{Eq:trPA}. 

    Restricting~\eqref{Eq:BracketFromDer} to $\sss\otimes\sss$ gives, with $D:=D_\delta|_{\sss}\in\Der(\sss)$,
    \begin{equation} \label{Eq:0-deg-bracket-from-delta}
        \{\overline a,\overline b\}=\overline a\,D(\overline b)-D(\overline a)\,\overline b.
    \end{equation}
    Equation~\eqref{Eq:D-on-Anat-from-delta} of Proposition~\ref{Pr:DerOA-delta} yields $D(\overline a)=\pi(\delta(a))$ on generators $\overline a\in\A_\natural$ of $\sss$. Substituting and rewriting each summand inside $\pi$,
    \begin{equation*}
        \{\overline a,\overline b\}=\overline a\,\pi(\delta(b))-\pi(\delta(a))\,\overline b=\pi\bigl(\delta(b)\cdot\overline a\bigr)-\pi\bigl(b\otimes\pi(\delta(a))\bigr)=\pi(d_a(b)),
    \end{equation*}
    which completes the proof.
\end{proof}

\begin{theorem}[\cite{procesi1976}] \label{Thm:trace-procesi}
    $\A_N^{\GLN}$ is generated as a $\kk$-algebra by $\tr(\overline a)$ for $a\in\A$.
\end{theorem}

The following theorem describes a transposed Poisson structure on $\A_N^{\GLN}$ from an $H_0$-transposed Poisson structure on $\A$.

\begin{theorem} \label{Pr:H0P-trace-tPA}
    Let $\A$ be a unital associative algebra equipped with an $H_0$-transposed Poisson structure $\{-,-\}$ in the sense of Definition~\ref{Def:trH0P}. Then, the operation
    \begin{equation} \label{Eq:trace-formula}
\{-,-\}_N^{\mathrm{tr}}\colon\A_N^{\GLN}\times\A_N^{\GLN} \longrightarrow\A_N^{\GLN},\qquad \{\tr(\overline a),\tr(\overline b)\}_N^{\mathrm{tr}}=\tr\bigl(\{\overline a,\overline b\}\bigr), 
    \end{equation}
    defines a transposed Poisson structure on $\A_N^{\mathrm{GL}_N}$.
\end{theorem}
\begin{proof}
    Proposition~\ref{Pr:tPA-unit} yields a unique derivation $D\in\Der(\sss)$, $D:=\{\mathds{1}_{\sss},-\}$, with
    \begin{equation} \label{Eq:trace-pf-D}
        \{f,g\}=f\,D(g)-D(f)\,g\qquad(f,g\in\sss).
    \end{equation}
    Suppose a transposed Poisson bracket $\{-,-\}_N^{\mathrm{tr}}$ on $\A_N^{\GLN}$ satisfies~\eqref{Eq:trace-formula}. Proposition~\ref{Pr:tPA-unit} applied to $\A_N^{\GLN}$ writes it as $\{x,y\}_N^{\mathrm{tr}}=xE(y)-E(x)y$ for a unique $E=\{1,-\}_N^{\mathrm{tr}}\in\Der_\kk(\A_N^{\GLN})$. Setting $a=1_\A$\footnote{Throughout the proofs of Theorem~\ref{Pr:H0P-trace-tPA} and Proposition~\ref{Pr:dtPA-N-matches-trace} we write $1_\A$ for the unit of $\A$, to avoid confusion with the units of the other algebras in play, such as $\mathds{1}\in\sss$ and $1\in\A_N^{\GLN}$.} in~\eqref{Eq:trace-formula} we obtain
    \begin{equation} \label{Eq:trace-pf-E-on-trace}
        E\bigl(\tr(\overline b)\bigr)=\tr\bigl(D(\overline b)\bigr)-c\,\tr(\overline b),\qquad c:=\tfrac{1}{N}\tr\bigl(D(\overline{1_\A})\bigr),\qquad b\in\A.
    \end{equation}
    Theorem~\ref{Thm:trace-procesi} shows that $\A_N^{\GLN}$ is generated by $\{\tr(\overline b)\}_{b\in\A}$, so a $\kk$-derivation of $\A_N^{\GLN}$ is determined by its values on these generators. The theorem reduces to producing $E\in\Der_\kk(\A_N^{\GLN})$ satisfying~\eqref{Eq:trace-pf-E-on-trace}; the bracket $xE(y)-E(x)y$ is then transposed Poisson by Example~\ref{Ex:baiDer}, and reversing the substitution above returns~\eqref{Eq:trace-formula}.

    \medskip
    Pick $d_{1_\A}\in(\Der(\A)\oplus\kk\cdot\id_\A)\otimes_\kk\sss$ as in Definition~\ref{Def:trH0P} for $a=1_\A$ the unit of $\A$, and write it as a finite sum
    \[
        d_{1_\A}=\sum_{i\in I}E_i\otimes f_i\;+\;\id_\A\otimes\kappa_{\sss},\qquad E_i\in\Der(\A),\ f_i,\kappa_{\sss}\in\sss.
    \]
    Each $E_i$ lifts to the unique $(E_i)_N\in\Der(\A_N)$ with $(E_i)_N(a_{kl})=(E_i(a))_{kl}$. Set $\kappa:=\tr(\kappa_{\sss})\in\A_N^{\GLN}$ and define
    \[
        \widehat{d_{1_\A}}\colon\A_N\longrightarrow\A_N,\qquad \widehat{d_{1_\A}}(x):=\sum_{i\in I}(E_i)_N(x)\,\tr(f_i)+\kappa\,x,
    \]
    which is $\GLN$-equivariant, so $\widehat{d_{1_\A}}$ restricts to a $\kk$-linear endomorphism of $\A_N^{\GLN}$.
    For each $E_i\in\Der(\A)$, $(E_i)_N(\tr(\overline b))=\sum_{k=1}^N(E_i(b))_{kk}=\tr(\overline{E_i(b)})$, so substituting into the definition of $\widehat{d_{1_\A}}$,
    \[
        \widehat{d_{1_\A}}\bigl(\tr(\overline b)\bigr)=\sum_{i\in I}\tr\bigl(\overline{E_i(b)}\bigr)\tr(f_i)+\kappa\,\tr(\overline b)=\tr\bigl(\pi(d_{1_\A}(b))\bigr)=\tr\bigl(\{\overline{1_\A},\overline b\}\bigr),
    \]
    where the second equality uses $\pi(c\otimes f)=\overline c\cdot f$ from Definition~\ref{Def:trH0P} and the third is~\eqref{Eq:trH0P}. By~\eqref{Eq:trace-pf-D} and the multiplicativity of $\tr$,
    \begin{equation} \label{Eq:trace-pf-d1A-on-trace}
        \widehat{d_{1_\A}}\bigl(\tr(\overline b)\bigr)=N\,\tr\bigl(D(\overline b)\bigr)-\tr\bigl(D(\overline{1_\A})\bigr)\,\tr(\overline b),\qquad b\in\A.
    \end{equation}

    Specialize~\eqref{Eq:trace-pf-d1A-on-trace} to $b=1_\A$: the right-hand side vanishes, so $\widehat{d_{1_\A}}(N)=0$. On the other hand $(E_i)_N(N)=0$ for $E_i\in\Der(\A)$, so by $\kk$-linearity $\widehat{d_{1_\A}}(N)=N\kappa$. Hence $\kappa=0$ and $\widehat{d_{1_\A}}\in\Der_\kk(\A_N)$; by equivariance it restricts to a derivation of $\A_N^{\GLN}$.

    Set $E:=\tfrac{1}{N}\widehat{d_{1_\A}}\big|_{\A_N^{\GLN}}\in\Der_\kk(\A_N^{\GLN})$. Identity~\eqref{Eq:trace-pf-d1A-on-trace} then reads
    \[
        E\bigl(\tr(\overline b)\bigr)=\tr\bigl(D(\overline b)\bigr)-c\,\tr(\overline b),
    \]
    which is~\eqref{Eq:trace-pf-E-on-trace} as desired.
\end{proof}

\subsection{Double transposed Poisson algebras and invariants} \label{ss:dtPA-Inv}

A double transposed Poisson bracket on $\A$ produces a transposed Poisson bracket on the invariant subalgebra $\A_N^{\GLN}$ in two different ways, and we show that they coincide. Throughout, let $\A$ be a unital associative algebra equipped with a double transposed Poisson algebra structure $\dsq{-,-}$ in the sense of Definition~\ref{Def:trDPA}.

First, Proposition~\ref{Pr:dtPA-induces-tPA-N} produces a $\GLN$-equivariant transposed Poisson bracket $[-,-]_N$ on $\A_N$, and Proposition~\ref{Pr:tPA-Red} restricts it to a transposed Poisson bracket
\[
    \{-,-\}_N^{\mathrm{ind}}\colon\A_N^{\GLN}\times\A_N^{\GLN}\longrightarrow\A_N^{\GLN}.
\]

Second, by Proposition~\ref{Thm:0-deg-is-H0P}, the restriction of $\dsq{-,-}$ to $\sss\otimes\sss$ is an $H_0$-transposed Poisson bracket $\{-,-\}$ on $\sss=\O(\A)^{(0)}$. Theorem~\ref{Pr:H0P-trace-tPA} associates with it the transposed Poisson bracket $\{-,-\}_N^{\mathrm{tr}}$
\eqref{Eq:trace-formula}. 

\begin{proposition} \label{Pr:dtPA-N-matches-trace}
    The two transposed Poisson brackets $\{-,-\}_N^{\mathrm{ind}}$ and $\{-,-\}_N^{\mathrm{tr}}$ on $\A_N^{\GLN}$ coincide.
\end{proposition}
\begin{proof}
    For $a\in\A$, the class $\overline a$ lies in $\sss=\O(\A)^{(0)}$ and has degree zero in $\O(\A)$. The pairing $(\overline a\,|\,1)$ of Section~\ref{Sec:Gen} with the unique element $1\in\Mat^{\otimes 0}=\kk$ is computed from~\eqref{f28} with empty index tuples $\mathbf{i},\mathbf{j}$, leaving only the trace factor:
    \[
        (\overline a\,|\,1)=\tr(\overline a)\in\A_N^{\GLN}.
    \]
    Formula~\eqref{Eq:induced-bracket-AN} for the induced bracket, specialized to $(\alpha,X)=(\overline a,1)$ and $(\beta,Y)=(\overline b,1)$, becomes
    \[
        \{\tr(\overline a),\tr(\overline b)\}_N^{\mathrm{ind}}=\bigl(\dsq{\overline a,\overline b}\,\big|\,1\bigr).
    \]
    Since $\dsq{\overline a,\overline b}$ lies in $\sss$, Proposition~\ref{Thm:0-deg-is-H0P} identifies it with $\{\overline a,\overline b\}$, and the pairing with $1$ of a degree-zero element is its trace; together with the defining formula~\eqref{Eq:trace-formula} of $\{-,-\}_N^{\mathrm{tr}}$, this gives
    \[
        \{\tr(\overline a),\tr(\overline b)\}_N^{\mathrm{ind}}=\tr\bigl(\{\overline a,\overline b\}\bigr)=\{\tr(\overline a),\tr(\overline b)\}_N^{\mathrm{tr}}.
    \]

    By Proposition~\ref{Pr:tPA-unit} applied to the unital commutative $\kk$-algebra $\A_N^{\GLN}$, each of $\{-,-\}_N^{\mathrm{ind}}$ and $\{-,-\}_N^{\mathrm{tr}}$ has the form $\{x,y\}=x\,E(y)-E(x)\,y$ for a unique $\kk$-derivation $E\in\Der_\kk(\A_N^{\GLN})$, recovered from the bracket by pairing with the multiplicative unit $1\in\A_N^{\GLN}$ via $E(y)=\{1,y\}$. By Theorem~\ref{Thm:trace-procesi}, $\A_N^{\GLN}$ is generated as a $\kk$-algebra by $\{\tr(\overline a)\}_{a\in\A}$. The previous paragraph shows that the two brackets agree on these generators; since $1=(1/N)\tr(\overline{1_\A})$ in $\A_N^{\GLN}$ and the bracket is $\kk$-bilinear, the associated derivations $E(y)=\{1,y\}$ agree on the generators as well. A $\kk$-derivation is determined by its values on a generating set, so the two derivations coincide, and so do the brackets.
\end{proof}